\documentclass[11pt]{amsart}
\usepackage{amsfonts,amssymb,amsthm,amsmath,amsxtra,amscd,verbatim,eucal}
\usepackage{mathrsfs}				    
\usepackage{enumitem}
\usepackage[all]{xy}
\usepackage[dvips]{graphics}
\usepackage{graphicx}
\usepackage{hyperref}
\usepackage{microtype}
\usepackage{verbatim}  

\usepackage{color}

\theoremstyle{plain}
\newtheorem{thm}{Theorem}[section]

\newtheorem{letthm}{Theorem}

\newtheorem{letcor}[letthm]{Corollary}

\newtheorem{lem}[thm]{Lemma}
\newtheorem{prop}[thm]{Proposition}

\newtheorem{cor}[thm]{Corollary}

\theoremstyle{definition}
\newtheorem{defi}[thm]{Definition}

\theoremstyle{remark}

\newtheorem{rmk}[thm]{Remark}

\definecolor{darkred}{RGB}{220,0,0}

\newcommand{\bb}[1]{{\mathbb{#1}}}
\newcommand{\mc}[1]{{\mathcal{#1}}}

\def\Z{\bb{Z}}
\def\N{\bb{N}}

\def\C{\bb{C}}
\def\A{\bb{A}}
\def\R{\bb{R}}
\def\Q{\bb{Q}}
\def\P{\bb{P}}

\def\cO{\mc{O}}
\def\cJ{\mc{J}}

\def\cE{\mathcal{E}}

\def\eps{\varepsilon}

\def\X{\mathscr X}
\def\Y{\mathscr Y}

\DeclareMathOperator{\ord} {ord}

\DeclareMathOperator{\GL}{GL}

\DeclareMathOperator{\id}{id}

\DeclareMathOperator{\NL}{NL}
\DeclareMathOperator{\LCep}{{L}_\C^\eps}

\DeclareMathOperator{\Spec}{Spec}
\DeclareMathOperator{\Spf}{Spf}
\DeclareMathOperator{\Div}{Div}
\DeclareMathOperator{\Dual}{Dual}
\DeclareMathOperator{\cent}{c}
\DeclareMathOperator{\For}{For}

\newcommand{\edge}[2]{{#1\!-\!\!\!-#2}}

\newcommand{\LC}[1]{{L}_\C^{#1}}

\title{Links of sandwiched surface singularities and self-similarity}

\date{\today}

\def\cO{\mc{O}}
\author[L. Fantini]{Lorenzo Fantini}
\address{Institut Math\'ematique de Jussieu, 
Universit\'e Pierre et Marie Curie, 
75252 Paris Cedex, France.}
\email{lorenzo.fantini@imj-prg.fr}

\author[C. Favre]{Charles Favre}
\address{CNRS - Centre de Math\'ematiques Laurent Schwartz, 
\'Ecole polytechnique, 
91128 Palaiseau Cedex, France.}
\email{charles.favre@polytechnique.edu}

\author[M. Ruggiero]{Matteo Ruggiero}
\address{Institut Math\'ematique de Jussieu, Universit\'e Paris 7, B\^atiment Sophie Germain, Case 7012, 75205 Paris Cedex 13, France.}
\email{matteo.ruggiero@imj-prg.fr}

\thanks{During the preparation of this article the first and the second authors have been supported by the ERC-starting grant project ``Nonarcomp'' (grant number 307856), and the second author by the ANR project ``D\'efig\'eo''.}

\begin{document}

\begin{abstract}
We characterize sandwiched singularities in terms of their  link in two different settings. We  first prove that such singularities are precisely the normal surface singularities having self-similar non-archimedean links.
We describe this self-similarity both in terms of Berkovich analytic geometry and of the combinatorics of weighted dual graphs.
We then show that a complex surface singularity is sandwiched if and only if its complex link can be embedded in a Kato surface in such a way that its complement remains connected.
\end{abstract}

\maketitle

\tableofcontents


\section{Introduction}
Let $X$ be any complex algebraic variety of dimension $d$ and let $0\in X$ be an isolated singular point.
A classical way  to analyze the geometry of $X$ near its singular point is to consider its (archimedean) \emph{link}, which is defined by embedding the complex analytic germ $(X,0)$ in the germ of a complex affine space $(\C^n,0)$ and taking the intersection with the boundary of a small ball around the origin.
More precisely, if $z_1,...,z_n$ are coordinates for $\C^n$ at zero, the intersection of $X$ with any sphere centered at $0$ of small enough radius $\varepsilon>0$ is transversal, so that 
$
\text{L}_\C^\varepsilon(X,0) = \big\{ x\in X(\C) \text{ s.t. } \sum_{i=1}^n|z_i(x)|_\C^2 = \varepsilon \big\}
$
is a smooth manifold of real dimension $2d-1$. 
Its diffeomorphism type does not depend on the embedding nor on $\varepsilon$, provided that $\varepsilon$ is small enough, and we define the  link of $(X,0)$ to be this diffeomorphism type. 
Note that the topology of a neighborhood of $0$ in $X$ is completely determined by its link, since one can show that the intersection of $X$ with a small ball is homeomorphic to the cone over  $\text{L}_\C^\varepsilon(X,0)$. 
The complex structure on $X$ also induces a canonical  contact structure on the link which has attracted a lot of attention recently, see for example \cite{CaubelNemethiPopescu-Pampu2006,McLean2015}.

When the algebraic variety $X$ is defined over an  algebraically closed field $k$, then
a non-archimedean version of the link can be defined as follows.
Endow $k$ with the trivial absolute value $|\cdot|$, that is the one such that $|k^\times|=1$, and denote by $X^\mathrm{an}$ the non-archimedean analytic space associated with $X$, in the sense of Berkovich~\cite{berkovich:book}.
Then, the space $\NL^\varepsilon(X,0) = \big\{ x\in X^\mathrm{an} \text{ such that} \max_{i}|z_i(x)| = \varepsilon \big\}$, with the topology induced from the one of $X^\mathrm{an}$, does not depend on the embedding nor on $\varepsilon\in\left]0,1\right[$.
We will call it the \emph{non-archimedean link} of $(X,0)$ and we will simply denote it by $\NL(X,0)$.
Observe that, thanks to the non-archimedean triangular inequality, the equation $\max_{i}|z_i(x)| = \varepsilon$ defines the boundary of the ball of radius $\varepsilon$ in the non-archimedean analytification of $\C^n$, making the definition of the non-archimedean link completely analogous to the classical one.
Concretely, $\NL(X,0)$ is the set of semi-valuations $v$ on the complete local ring $\widehat{\mathcal O_{X,0}}$ of $X$ at $0$ that are normalized by the condition $\min_{f\in{\mathfrak M}}v(f)=1$, where $\mathfrak M$ is the maximal ideal of $\widehat{\mathcal O_{X,0}}$, endowed with the pointwise convergence topology.

The homotopy type of the non-archimedean link $\NL(X,0)$ is well understood in terms of the resolutions of singularities of $(X,0)$, whenever those exist.
Recall that any resolution of singularities $\pi\colon X_\pi\to X$ of $(X,0)$ whose exceptional divisor $\pi^{-1}(0)$ has simple normal crossing singularities gives rise to a \emph{dual simplicial complex} $\Delta_\pi$, which is a finite simplicial complex encoding the incidence relations between the components of $\pi^{-1}(0)$.
It follows from the work of Thuillier \cite{thuillier:geometrietoroidale} that $\Delta_\pi$ can be embedded in $\NL(X,0)$ and that there is a deformation retraction of the latter onto the former.
Since every connected finite simplicial complex is the dual complex of an isolated normal singularity by Koll\'ar \cite{kollar:links}, the homotopy type of $\NL(X,0)$ can be arbitrarily complicated. 
However, de Fernex--Koll\'ar--Xu \cite{defernex-kollar-xu:dualcomplex} have proved that $\Delta_\pi$ is contractible for isolated log terminal singularities.

On the other hand, the topology of non-archimedean links is poorly understood and has been analyzed in depth only in the case of surfaces. 
One can show that $\NL(\A^2_k,0)$ is a compact real tree, that is a union of segments which does not contain any non-trivial loop, see~\cite{berkovich:book,jonsson:berkovich,favre-jonsson:valtree}. 
Its structure is however quite intricate since it has a dense set of ramification points (corresponding to points of \emph{type 2} as in~\cite{berkovich:book}), and the set of branches at such a point
is naturally parameterized by $\P^1(k)$, which may be uncountable. When $k$ is a countable field,  $\NL(\A^2_k,0)$ is metrizable and homeomorphic to the  Wa\.{z}ewski universal dentrite by~\cite{HrushovskiLoeserPoonen2014}. 

The non-archimedean link of a surface singularity $(X,0)$ can be obtained by gluing copies of $\NL(\A^2_k,0)$ to a finite graph. This picture has enabled   
de Felipe \cite{deFelipe2017} to completely describe the homeomorphism types of non-archimedean links of surface singularities, but
her result shows that the topology of $\NL(X,0)$ fails to encode much information about the singularity. For example, $\NL(X,0)$ is homeomorphic to $\NL\big(\A^2_k,0\big)$ when $(X,0)$ is a rational singularity, or when $X$ admits a good resolution whose exceptional locus is irreducible. The topology of $\NL(X,0)$ thus forgets the function fields of exceptional components of a resolution. In order to characterize interesting classes of singularities one needs to retain some of this information.

%
%
%

To do so, we consider the sheaf on $\NL(X,0)$ which is induced by the sheaf of analytic functions on $X^\mathrm{an}$.
The resulting ringed space was studied by the first author in \cite{fantini:normspaces}.
Note that 
we cannot expect $\NL(X,0)$ with this additional analytic structure to be isomorphic to a proper subspace of itself, as any such isomorphism would have to send the endpoints of $\NL(X,0)$ to endpoints, forcing it to be surjective.
In particular, already in the case of a smooth point in a surface the non-archimedean link is isomorphic to a proper subspace of itself only after removing finitely many endpoints (more precisely, finitely many points of \emph{type 1}, that are the endpoints corresponding to semi-valuations with nontrivial kernel).
Whenever such an isomorphism exists we will say that the non-archimedean link $\NL(X,0)$ is \emph{self-similar} (see the condition \ref{condition_prime} on page~\pageref{conditiondag}).

In this paper we show that the non-archimedean link $\NL(X,0)$ of a normal surface singularity $(X,0)$ is self-similar if and only if $(X,0)$ is a \emph{sandwiched singularity}.
Over the complex numbers, sandwiched singularities were defined by Spivakovisky in \cite{spivakovsky:sandsingdesingsurfNashtransf} as those normal surface singularities whose complex analytic germs dominate bimeromorphically a smooth germ; in \emph{loc. cit.} they play a crucial role in the proof of the desingularization of surfaces via Nash transformations.
Several authors have further contributed to the study of sandwiched singularities, for example their deformation theory has been investigated by T. de Jong and van Straten \cite{deJongvanStraten1998}, while their Milnor fibers have been described by N\'emethi and Popescu-Pampu \cite{NemethiPopescu-Pampu2010}.
In order to work over an algebraically closed field $k$ of arbitrary characteristic we will need to replace complex analytic germs with formal germs, that is we will work with complete local rings; a precise definition will be given in Section~\ref{section_preliminariessandwiched}.

Our main result is the following theorem which give several characterizations of sandwiched singularities.

\begin{letthm}\label{mainthm}
Let $(X,0)$ be a normal surface singularity.
The following are equivalent:
\begin{enumerate}[label=(\roman{enumi})]
\item \label{condition_sandwiched}
$(X,0)$ is sandwiched;

\item \label{condition_strongly_selfsim}
there exists a finite set $T$ of type 1 points of $\NL(X,0)$ such that every point of $\NL(X,0)$ that is not of type 2 has a basis of neighborhoods each isomorphic to $\NL(X,0)\setminus T$;

\item \label{condition_valtree}
there exists a finite set $T$ of type 1 points of $\NL(X,0)$ such that $\NL(X,0)\setminus T$ is isomorphic to an open subspace of $\NL(\A^2_k,0)$;

\item \label{condition_selfsim}
there exists a finite set $T$ of type 1 points of $\NL(X,0)$ such that every open subset of $\NL(X,0)$ contains an open subset isomorphic to $\NL(X,0) \setminus T$;

\item \label{condition_graph}
there exists a good resolution of $(X,0)$ whose weighted dual graph is self-similar;

\item \label{condition_kato}
there exists a proper birational morphism of algebraic $k$-surfaces $\pi\colon X'\to X$ which is not an isomorphism above $0$, together with a point $p\in \pi^{-1}(0)$ and an isomorphism of complete local rings $\widehat{\mathcal O_{X',p}} \cong \widehat{\mathcal O_{X,0}}$.
\end{enumerate}
\end{letthm}

Both \ref{condition_strongly_selfsim} and \ref{condition_selfsim} can be interpreted as self-similarity properties for $\NL(X,0)$, and imply the condition~\ref{condition_prime}.
In \ref{condition_graph}, a vertex of a dual graph has as weights the genus and the self intersection of the corresponding component; such a weighted graph is said to be \emph{self-similar} if it is isomorphic to a graph modification of itself, see Section~\ref{section_graphs} for more details.
A datum like the one of \ref{condition_kato} will be called a \emph{Kato datum} for $(X,0)$.

Let us now illustrate the main ingredients of the proof of Theorem~\ref{mainthm}, which requires a combination of methods from resolution of singularities, non-archimedean analytic geometry, formal geometry, and combinatorics.

Begin by observing that any sandwiched singularity can be obtained by performing a composition of point blowups $(Y,D) \to (\A^2_k,0)$, followed by the contraction of a connected divisor $E$ on $Y$ that is supported on $D$.
This procedure yields two maps $Y\to X \to \A^2_k$ that shows that the contracted surface $X$, whose singular point $0$ is the image of $E$ through the contraction map, is ``sandwiched'' between two smooth surfaces, justifying the terminology.
By picking any point $y$ in $E$ and performing on it the same sequence of blowups and contraction (see the subsection~\ref{sec:164} for a detailed explanation of how this can be done), one obtains a surface $X'$ above $Y$ with a singular point $p$ such that $\widehat{\mathcal O_{X',p}} \cong \widehat{\mathcal O_{X,0}}$, that is a Kato datum. This proves the implication  \ref{condition_sandwiched}  $\implies$ \ref{condition_kato}.

In the sections \ref{section_preliminarieslink} and \ref{section_links_surfaces} we study in detail the structure of non-archimedean links.
In particular, with any modification $(Y',D')$ of $(X,0)$ is associated a \emph{center map} $\cent_{Y'}\colon \NL(X,0)\to D'$, and if $y'$ is a closed point of $D'$ then its inverse image $\cent_{Y'}^{-1}(y')$ is an open subspace of $\NL(X,0)$ that is isomorphic to the complement of finitely many points of type 1 in the non-archimedean link $\NL(Y',y')$ of $y'$ in $Y'$ (see Proposition~\ref{proposition_propertiesNL}).
When applied to the Kato datum $X'\to X$ above, this shows that $\NL(X,0)$ contains a strict open subspace that is isomorphic to the complement of finitely many points of type 1 in $\NL(X,0)$ itself, since $\widehat{\mathcal O_{X',p}} \cong \widehat{\mathcal O_{X,0}}$ implies that $\NL(X',p)$ and $\NL(X,0)$ are isomorphic, that is the condition~\ref{condition_prime} holds. 
Similar arguments based on the study of the structure of non-archimedean links permit to obtain the implications \ref{condition_sandwiched} $\implies$ \ref{condition_valtree}   $\implies$ \ref{condition_strongly_selfsim} $\implies$ \ref{condition_selfsim} of Theorem~\ref{mainthm}, as is explained in Section~\ref{section_preliminariessandwiched}.

Showing that singularities having self-similar links have also self-similar weighted dual graphs, that is the implication \ref{condition_prime} $\implies$ \ref{condition_graph} of Theorem~\ref{mainthm}, is a slightly more delicate matter, undertaken in Section~\ref{section_proof_main}.
We prove an extension result for morphisms of a punctured disc into $\NL(X,0)$, Proposition~\ref{lem:keyboundary}, and deduce that we can assume that the boundary $\partial U$ of an open subset $U$ of $\NL(X,0)$ isomorphic to the complement of finitely many points of type 1 in $\NL(X,0)$ consists only of finitely many points of type 2.
We then use other results on the structure of $\NL(X,0)$, proven in subsection \ref{subsection_existencemodifications}, to produce a (formal) modification $(Y',D')$ of $(X,0)$ together with a closed point $y'$ of $D'$ such that $U\cong \cent_{Y'}^{-1}(y)$.
Finally, we show that the dual graph associated with $(X,0)$ is self-similar by carefully choosing compatible resolutions of $(Y',y)$ and $(X,0)$.

The proof of the remaining implication, that is \ref{condition_graph} $\implies$ \ref{condition_sandwiched}, requires two distinct steps that we believe to be of independent interest.
First of all, we prove a purely combinatorial result, Theorem~\ref{thm:graph-sdw}, showing that every self-similar weighted graph is \emph{sandwiched}, which means that it can be embedded in a graph modification of the trivial graph (the dual graph of the blowup of $\A^2_k$ at $0$).
We then show in Theorem~\ref{thm:extend-spiv} that, if $(X,0)$ is a normal surface singularity admitting a good resolution whose associated weighted dual graph is sandwiched, then $(X,0)$ is a sandwiched singularity.
Over the complex numbers this result was originally proven by Spivakovsky in \cite{spivakovsky:sandsingdesingsurfNashtransf}, but his proof relies on plumbing techniques for complex analytic spaces; we proceed in a similar way, using an analogue of plumbing in formal geometry.

\medskip
\begin{center}
$\diamond$
\end{center}

\medskip

Since sandwiched singularities can be characterized in terms of their non-archimedean links, it is also natural to look for a characterization of them in terms of their archimedean links. 
We have not been able to find a self-similar property reminiscent of Theorem~\ref{mainthm}, \ref{condition_strongly_selfsim} or \ref{condition_selfsim}. 
However, building on Theorem~\ref{mainthm}, \ref{condition_kato}, in Section~\ref{section_complexanalytic} we observe that links of sandwiched singularities are exactly those arising on a specific class of smooth compact complex surfaces.

To state precisely our results, we need to introduce some terminology.  A compact complex surface $S$ contains \emph{a global spherical shell} if it admits a biholomorphic copy of a neighborhood of the $3$-sphere in $\C^2$ that does not disconnect $S$. Surfaces containing a global spherical shell have been completely described by Kato~\cite{kato:cptcplxmanifoldsGSS} (see also the subsequent work of G.~Dloussky~\cite{dloussky:phdthesis}).
They are non-k\"ahler compact surfaces of Kodaira dimension equal to $-\infty$ that play a special role in the Kodaira classification of compact complex  surfaces, see the introduction of \cite{MR2726099}.
Primary Hopf surfaces are the most emblematic examples of such surfaces: they are obtained as the orbit space of a contracting germ of biholomorphism of $(\C^2,0)$, and are diffeomorphic to the product of spheres $\mathbb S^3\times \mathbb S^1$. Any surface containing a global spherical shell is a deformation of a modification of a primary Hopf surface.

\begin{letthm}\label{thm3}
Let $(X,0)$ be a complex sandwiched 
singularity, and choose a local embedding $X \subset \C^n$.  
Then, for any $\eps$ small enough, there exist a smooth compact complex surface $S$ having a global spherical shell and a holomorphic embedding 
\[
\imath : X\cap \{z\in \C^n,\, \eps/2 < \|z\| < 2 \eps\} \longrightarrow S
\]
such that $S \setminus \LCep(X,0)$ is connected. 
\end{letthm}

Observe that the link $\LCep(X,0)$ is included in the domain $X\cap \{z\in \C^n,\, \eps/2 < \|z\| < 2 \eps\}$ so that  one can phrase the property in the previous theorem as an embedding property for the archimedean link of a sandwiched singularity.  Following the terminology of M.~Kato, this says that 
$\LCep(X,0)$ can be realized as a real-analytic global strongly pseudoconvex $3$-fold 
in a surface containing a global spherical shell, see Section~\ref{section_complexanalytic} for a discussion of these notions.

We prove that this property characterizes sandwiched singularities.

\begin{letthm}\label{thm2}
Let $(X,0)$ be a complex  normal surface singularity, and choose a local embedding $X \subset \C^n$. 
Suppose that for some small enough $\eps>0$, the archimedean link $\LCep (X,0)$ can be realized as a real-analytic global strongly pseudoconvex $3$-fold in a compact complex surface $S$.
Then we are in exactly one of the following situations:
\begin{enumerate}[label=(\roman{enumi})]
\item\label{item:thm2a}
$(X,0)$ is a weighted homogeneous singularity which is not rational, and $S$ is an elliptic surface of Kodaira dimension either $0$ or $1$;
\item\label{item:thm2b}
$(X,0)$ is a quotient singularity, and $S$ is a secondary Hopf surface; 
\item\label{item:thm2c}
$(X,0)$ is a sandwiched singularity, and $S$ carries a global spherical shell.
\end{enumerate}
\end{letthm}
An elliptic surface is a compact complex surface carrying an elliptic fibration~\cite[p.200]{barth-hulek-peters-vanderven:compactcomplexsurfaces}. A 
Hopf surface is a compact complex surface whose universal cover is biholomorphic  to $\C^2\setminus\{0\}$. When its fundamental group is cyclic, then it is a primary Hopf surface in the previous sense. 
Otherwise it admits a non-trivial finite cyclic cover by a primary Hopf surface, in which case it is called secondary Hopf surface.

A secondary Hopf surface is never elliptic and does not contain any global spherical shell, so that the three classes of surfaces arising in Theorem~\ref{thm2} are really disjoint.


Observe that Theorems~\ref{thm3} and~\ref{thm2} imply immediately the following result.

\begin{letcor}\label{cor2}
A normal surface singularity is sandwiched if and only if its archimedean link can be realized as a 
real-analytic global strongly pseudoconvex $3$-fold in a compact complex surface containing a global spherical shell.
\end{letcor}

\noindent {\bf Acknowledgements.} We would like to warmly thank B. Teissier, who asked us about a characterization of singularities having self-similar Riemann-Zariski spaces. 
This paper is a tentative answer to his question in the framework of normalized Berkovich analytic spaces.

\section{Preliminaries on non-archimedean links}\label{section_preliminarieslink}

In this section we recall the construction of the non-archimedean link $\NL(X,Z)$ of a subscheme $Z$ in an algebraic variety $X$ from \cite{fantini:normspaces} (where it is called normalized non-archimedean link).

\subsection{Berkovich analytifications}
We begin by recalling the definition of the Berkovich analytification of an algebraic variety, following \cite{berkovich:book}.
Let $K$ be a field complete with respect to a non-archimedean absolute value $|\cdot|$, that is an absolute value such that $|a+b|\leq \max\{|a|,|b|\}$ for every $a$ and $b$ in $K$.
We denote by $K^\circ  =\{ |a| \le 1\}$ the valuation ring of $K$. 
In the sequel $K$ will either be the field of Laurent series $k((t))$ with a $t$-adic absolute value for some algebraically closed field $k$, 
or any field endowed with the trivial absolute value, that is the absolute value such that $|K^\times|=1$.

Let $X=\Spec(A)$ be an affine algebraic variety over $K$. 
The \emph{analytification} $X^\mathrm{an}$ of $X$ is defined as the following set of multiplicative semi-norms:
\[
X^{\mathrm{an}}=\Big\{x\colon A\to \R_{\geq0} \,\big|\, x(ab)=x(a)x(b), x(a+b)\leq x(a)+x(b), x(c)=|c| \, \forall a,b\in A, c\in K\Big\},
\]
with the topology of the pointwise convergence, that is the topology induced by the product topology of $(\R_{\geq0})^A$.
The definition extends by gluing to any algebraic variety over $K$ (and more generally to any $K$-scheme of finite type).

If $x$ is a point of $X^\mathrm{an}$ then its kernel $\ker x$ is a prime ideal of $A$ and $x$ induces an absolute value on the quotient $A/\ker x$, which is the residue field of $X$ at $\ker x$. The completion of $A/\ker x$ with respect to this absolute value is a complete valued field extension of $K$ which will be denoted by $\mathscr H(x)$ and called the \emph{complete residue field} of $X^\mathrm{an}$ at $x$.
If $f$ is an algebraic function on $X$, we will denote by $f(x)$ its image in $\mathscr H(x)$, and by $|f(x)|\in\R_{\geq0}$ the image of $f(x)$ through the absolute value of $\mathscr H(x)$.

More generally, $X^\mathrm{an}$ comes equipped with a sheaf of $K$-algebras of \emph{analytic functions}, consisting of the functions 
which can be written locally as a uniform limit of rational functions without poles.
If $U$ is an open subspace of $X^\mathrm{an}$, then an analytic function $f\in\mathcal O_{X^\mathrm{an}}(U)$ can be evaluated in a point $x$ of $U$, yielding an element $f(x)\in \mathscr H(x)$, and therefore a positive real number $|f(x)|$.
An analytic funtion $f$ on $U$ is said to be \emph{bounded} if $|f(x)|\leq1$ for every $x$ in $U$.
Bounded analytic functions form a subsheaf $\mathcal O^\circ_{X^\mathrm{an}}$ of $K^\circ$-modules  of $\mathcal O_{X^\mathrm{an}}$.

Moreover, $X^\mathrm{an}$ can be endowed with an additional Grothendieck topology.
We will not discuss this last aspect in the rest of the paper since we will only be considering open subspaces of analytic spaces.

\subsection{The center map}
 For the remaining of this section we work with an algebraic variety $X$ over a trivially valued and algebraically closed field $k$.
 
 
Any point $x$ of $X^\mathrm{an}$ comes together with a morphism $\alpha\colon \Spec\big(\mathscr H (x)\big)\to X$.
We say that $x$ \emph{has center} on $X$ if $\alpha$ extends to a morphism $\overline{\alpha} \colon \Spec\big(\mathscr H (x)^\circ\big)\to X$, where $\mathscr H (x)^\circ$ is the valuation ring of $\mathscr H (x)$.
The \emph{center} of $x$ is then defined as the point $\mathrm{c}_X(x)=\overline{\alpha}(s)$ of $X$, where $s$ is the closed point of $\Spec\big(\mathscr H (x)^\circ\big)$.
This coincides with the notion, classical in valuation theory, of the center of the valuation ring $\mathscr H(x)^\circ$.
Observe that, since $X$ is separated by hypothesis, whenever $x$ has center on $X$ then its center $\mathrm{c}_X(x)$ is a well defined point of $X$.
If moreover $X$ is proper over $k$, then by the valuative criterion of properness every point of $X^\mathrm{an}$ has center on $X$, but this is not true in general.

The \emph{center map} $\mathrm c_X\colon \big\{x\in X^\mathrm{an} \text{ having center on } X\}\to X$ is anticontinuous, which means that $\mathrm c_X^{-1}(Z)$ is open whenever $Z$ is a closed subvariety of $X$.
The notion of center describes the property for a point of $X^\mathrm{an}$ to be close to a point of $X$, so whenever $Z$ is a closed subvariety of $X$ the subset $\mathrm c_X^{-1}(Z)$ of $X^\mathrm{an}$ can be thought of as a tubular neighborhood of $Z$ in $X^\mathrm{an}$.
The complement $\mathrm c_X^{-1}(Z)\setminus Z^\mathrm{an}$ of $Z^\mathrm{an}$ in the tubular neighborhood can then be thought of as a \emph{punctured tubular neighborhood} of $Z$ in $X^\mathrm{an}$.

%
\begin{rmk}\label{rmk:thu}
In the terminology of \cite{thuillier:geometrietoroidale}, the punctured tubular neighborhood $\mathrm c_X^{-1}(Z)\setminus Z^\mathrm{an}$ coincides with the analytic space $\big(\widehat{X/Z}\big)_\eta$ associated with the formal completion $\widehat{X/Z}$ of $X$ along $Z$.
\end{rmk}

\subsection{Non-archimedean links}
We fix an algebraic variety $X$ over $k$, and a nonempty and nowhere dense closed subscheme $Z$ of $X$. 

\smallskip

An element $\lambda$ of $\R_{>0}$ acts on a point $x$ of $X^\mathrm{an}$ by raising the semi-norm $x$ to the power $\lambda$.
Indeed, the condition for $x^\lambda$ to be in $X^\mathrm{an}$, that is the fact that it is trivial on $k$, is satisfied since the trivial absolute value is invariant under exponentiation by elements of $\R_{>0}$.
Observe that as an abstract field the completed residue field $\mathscr H(x)$ is isomorphic to $\mathscr H(x^\lambda)$, but the absolute value of the latter is obtained by raising to the power $\lambda$ the one of the former. 
Therefore, neither the abstract valuation ring $\mathscr H(x)^\circ$ nor the morphism of schemes $\Spec\big(\mathscr H(x)\big)\to X$ associated with $x$ change when replacing $x$ by $x^\lambda$.
It follows that the action of $\R_{>0}$ on $X^\mathrm{an}$ induces an action on the punctured tubular neighborhood $\mathrm c_X^{-1}(Z)\setminus Z^\mathrm{an}$.

\smallskip

We define the \emph{non-archimedean link} $\NL(X,Z)$ of $Z$ in $X$ as the quotient of the punctured tubular neighborhood of $Z$ in $X^\mathrm{an}$ by this action:
\[
\NL(X,Z) = \big( \mathrm{c}_X^{-1}(Z)\setminus Z^{\mathrm{an}} \big) \big/ \R_{>0}.
\]

We endow $\NL(X,Z)$ with the quotient topology, and  see it as a ringed space by endowing it with the following two sheaves. 
The \emph{sheaf of analytic functions} $\mathcal O_{\NL(X,Z)}$ on $\NL(X,Z)$ is by definition the push-forward to $\NL(X,Z)$ of the sheaf of analytic functions $\mathcal O_{\mathrm c_X^{-1}(Z)\setminus Z^\mathrm{an}}$ on the Berkovich analytic space $\mathrm c_X^{-1}(Z)\setminus Z^\mathrm{an}$ via the quotient map.
Analogously, the \emph{sheaf of bounded analytic functions} $\mathcal O_{\NL(X,Z)}^\circ$ on $\NL(X,Z)$ is the push-forward of the sheaf of bounded analytic functions $\mathcal O_{\cent_X^{-1}(Z)\setminus Z^\mathrm{an}}^\circ$.
Both are local sheaves of $k$-algebras and $\mathcal O_{\NL(X,Z)}^\circ$ is a subsheaf of $\mathcal O_{\NL(X,Z)}$.
We will say more about the analytic structure of $\NL(X,Z)$ in the subsequent subsections.
Observe that $\NL(X,Z)$ is a compact topological space by \cite[Proposition 5.9]{fantini:normspaces}.

It is worth noticing that by Remark~\ref{rmk:thu} the space $\mathrm c_X^{-1}(Z)\setminus Z^{\mathrm{an}} $ only depends on the formal completion of $X$ along $Z$, so the same is true for $\NL(X,Z)$.
In particular, if $0$ and $0'$ are closed points of algebraic $k$-varieties $X$ and $X'$ respectively, then the non-archimedean links $\NL(X,0)$ and $\NL(X',0')$ are isomorphic (as locally ringed spaces) if and only if the corresponding complete local rings $\widehat{\mathcal O_{X,0}}$ and $\widehat{\mathcal O_{X',0'}}$ are isomorphic.
It follows that if $k=\C$ then analytically isomorphic singularities have isomorphic non-archimedean links.

One important feature of non-archimedean links is their invariance under modifications.
We use the following non-conventional terminology: we define a \emph{modification} of $(X,Z)$ to be a pair $(Y,D)$, where $Y$ is a normal algebraic $k$-variety and $D$ is a Cartier divisor of $Y$, together with a proper morphism $\pi\colon Y\to X$ which is an isomorphism out of $D$ and such that $D=\pi^{-1}(Z)$ is the (schematic) inverse image of $Z$ through $\pi$.

\medskip

The following result is then a  consequence of the valuative criterion of properness.
\begin{prop}[{\cite[Proposition 1.11]{thuillier:geometrietoroidale}}]\label{proposition_invariance_modifications}
If $\pi\colon(Y,D)\to(X,Z)$ is a modification of $(X,Z)$, then $\pi$ induces an isomorphism of locally ringed spaces $\NL(Y,D)\stackrel{\sim}{\to}\NL(X,Z)$.
\end{prop}

The map $\mathrm{c}_X$ induces an anticontinuous map $\mathrm{c}_X\colon \NL(X,Z)\to Z$ which we still call the \emph{center map}.
Thanks to the result above, we also have a center map $\mathrm c_Y\colon \NL(X,Z)\to D$ associated with any modification $(Y,D)$ of $(X,Z)$.


\subsection{Analytic structure of non-archimedean links}
We fix an algebraic variety $X$ over $k$, and a nonempty and nowhere dense closed subscheme $Z$ of $X$. 
We now explain some properties of the ringed space structure of the non-archimedean link $\NL(X,Z)$.

\smallskip

Some caution is needed when working with analytic functions on $\NL(X,Z)$, as they cannot be evaluated at points of $\NL(X,Z)$, given that such a point is only a $\R_{>0}$-equivalence class of semi-norms.

However, it is possible to say whether the value of a function $f\in\mathcal O_{\NL(X,Z)}(U)$ at a point $x\in U$ lies in $\{0\}$, $]0,1[$, $\{1\}$, or $]1,+\infty[$, as those are the orbits of $\R_{\geq0}$ under the action of $\R_{>0}$ by exponentiation.
In particular it makes sense to ask whether a function vanishes at a point. 
Moreover, for any open set $U$ one can interpret $\mathcal O_{\NL(X,Z)}^\circ(U)$  as the ring of those functions on $U$ which are bounded by $1$.

\medskip

As follows from the definition, with any point $x$ of $\NL(X,Z)$ is associated a field extension $\mathscr H(x)$ of $k$, endowed with a rank $1$ valuation (but not with an absolute value) trivial on $k$, with respect to which $\mathscr H(x)$ is complete, and therefore a rank $1$ valuation ring which we still denote by $\mathscr H(x)^\circ$. 
Conversely, every complete (rank 1) valuation ring on the residue field of a scheme-theoretic point of $X\setminus Z$ such that its center on $X$ is in $Z$ comes from a point of $\NL(X,Z)$.
Observe that a function $f\in\mathcal O_{\NL(X,Z)}(U)$ is bounded by $1$ on $U$ if and only if $f(x)\in \mathscr H(x)^\circ$ for every $x$ in $U$.
The valued field $\mathscr H(x)$ can be defined more intrinsically, in a way that only depends on the ringed space structure of $\NL(X,Z)$, as the completion of the residue field of the local ring $\mathcal O_{\NL(X,Z),x}$ with respect to the valuation induced by $\mathcal O_{\NL(X,Z),x}^\circ$.

The analytic structure of $\NL(X,0)$ contains abundant information about the pair $(X,0)$, as is clear from the following result of \cite{fantini:normspaces}, which we recall for the reader's convenience.

\begin{prop}[{\cite[Corollary 4.14]{fantini:normspaces}}]\label{proposition_propertiesNL_1}
Let $0$ be any closed point in an algebraic variety $X$ and assume that $X$ is normal at $0$. 
Then the canonical morphism
\[
\widehat{\mathcal O_{X,0}}
\to
\mathcal O_{\NL(X,0)}^\circ \big( \NL(X,0) \big)
\]
is an isomorphism.
\end{prop}



\subsection{Non-archimedean links and $k((t))$-analytic spaces}
A crucial property of non-archimedean links is that  they are locally isomorphic to analytic spaces over a field of Laurent series $k((t))$ with $t$-adic absolute value, in the following sense.
Choose $\varepsilon\in]0,1[$ and endow $k((t))$ with the $t$-adic absolute value such that $|t|=\varepsilon$.
As we have seen, any Berkovich analytic space $\mathfrak X$ over $k((t))$, for example an open or closed subspace of the analytification of an algebraic $k((t))$-variety, comes equipped both with a sheaf of $k((t))$-algebras $\mathcal O_{\mathfrak X}$ and with a sheaf of $k[[t]]$-algebras $\mathcal O_{\mathfrak X}^\circ$.
We can see these two sheafs only as sheaves in $k$-algebras, yielding a triple which we denote by $\For(\mathfrak X)=\big({\mathfrak X},\mathcal O_{\mathfrak X},\mathcal O^\circ_{\mathfrak X}\big)$.
It then makes sense to ask whether a non-archimedean link is isomorphic as such a triple to a triple of the form $\For(\mathfrak{X})$.
In general, this is true only locally, in the following sense.
If $X$ is affine, then $\NL(X,Z)$ can be covered by finitely many open subspaces which, as ringed spaces in $k$-algebras, are of the form $\For(\mathfrak X)$ for some $k((t))$-analytic space $\mathfrak X$.
Observe that this is also the case if $Z$ is a single point of $X$, since $X$ can then be replaced by an affine neighborhood of $Z$.
If $X$ is not affine, then $\NL(X,Z)$ is covered by the compact domains $\mathrm c_X(U\cap Z)$, for $U$ ranging among an open affine cover of $X$, and each $\mathrm c_X(U\cap Z)$ is locally isomorphic to a $k((t))$-analytic space in the sense above.
Mind that the datum consisting of such a covering and $k((t))$-analytic structures is non canonical. 

A proof of this fact can be found in \cite[Corollary 4.10]{fantini:normspaces}, but to help the reader familiarize with the structure of non-archimedean links we illustrate here what happens in the case when $Z=\{0\}$ is a closed point of $X$.
Let $f$ be an element of the completed local ring $\widehat{\mathcal O_{X,0}}$ of $X$ at $0$.
This defines a $k$-analytic map from $\mathrm{c}_X^{-1}(0)$ into the open unit ball in ${\mathbb A}^{1,\mathrm{an}}_k$. 
The latter is canonically homemorphic to the interval $[0,1[$, and under
this homeomorphism this analytic map is given by the absolute value $\mathrm{c}_X^{-1}(0) \stackrel{|f|}{\longrightarrow} \left[0,1\right[$. 
%
The fiber $|f|^{-1}(\varepsilon)$ of $|f|$ at $\varepsilon \in \left]0,1\right[$ is then an analytic space over $k((t))$, since the completed residue field of $\mathbb A^{1,\mathrm{an}}_k$ at the point corresponding to $\varepsilon$ is the field $k((t))$ with the $t$-adic absolute value such that $|t|=\varepsilon$.

Then the projection $\pi\colon \mathrm{c}_X^{-1}(0)\setminus\{0\}^\mathrm{an} \to \NL(X,0)$ defining $\NL(X,0)$ identifies $\For\big(|f|^{-1}(\varepsilon)\big)$ with its image in $\NL(X,0)$, which is the complement $\NL(X,0)\setminus V(f)$ of the zero locus $V(f)$ of $f$ in $\NL(X,0)$.
Finally, by having $f$ range among a finite set of generators of the maximal ideal of $\widehat{\mathcal O_{X,0}}$, we obtain a cover $\NL(X,0)$ with finitely many open subspaces, each of which is isomorphic to a $k((t))$-analytic space.

\begin{rmk}
Observe that the $k((t))$-analytic space $|f|^{-1}(\varepsilon)$ is the analytic Milnor fiber $\mathscr F_{f,0}$ of $f$ at $0$, an object defined and studied in \cite{nicaise-sebag:analyticmilnorfiber}.
The non-archimedean link $\NL(X,0)$ can be thought of as a global version of $\mathscr F_{f,0}$, dependent only on the germ of $X$ at $0$ and not on $f$.
\end{rmk}

\subsection{Discs and annuli}\label{sec:DSK}

Some specific subspaces of non-archimedean links of surfaces are particularly important and deserve to be studied in depth.
Let $T$ denote a coordinate function on the $k((t))$-analytic affine line  $\A^{1,\mathrm{an}}_{k((t))}$.
We say that a subspace $U$ of a non-archimedean link $\NL(X,Z)$ is a \emph{disc} if, as a ringed space in $k$-algebras, $U$ is isomorphic to $\For(D)$, where $D=\big\{x\in \A^{1,\mathrm{an}}_{k((t))} \,\big|\, |T(x)|<1\big\}$ is an open unit $k((t))$-analytic disc.
We say that $U$ is an \emph{annulus} if, as a ringed space in $k$-algebras, $U$ is isomorphic to $\For(A)$, where $A=\big\{x\in \A^{1,\mathrm{an}}_{k((t))} \,\big|\, |t|<|T(x)|<1\big\}$ is an open $k((t))$-analytic annulus of modulus one.
We collect in the following statement some well known results about the topology of discs and annuli. 

\begin{thm}[See {{\cite[\S4.1 and \S4.2]{berkovich:book}}}]
	\label{thm:R-tree}
Let $D=\big\{x\in \A^{1,\mathrm{an}}_{k((t))} \,\big|\, |T(x)|<1\big\}$ and $A=\big\{x\in \A^{1,\mathrm{an}}_{k((t))} \,\big|\, |t|<|T(x)|<1\big\}$ be a $k((t))$-analytic disc and a $k((t))$-analytic annulus respectively.
For any $r\in (0,1)$ (respectively for $r\in (|t|,1)$), denote by $x_r$ the point of $D$ (respectively of $A$) defined by $|P(x_r) | = \sup_{|z|\le r} |P(z)|$ for any $P\in k((t))[T]$. \label{def of xr}
\begin{enumerate}
\item
$D$ and $A$ are uniquely arcwise connected: any two distint points $x, y$ are included in a unique closed subset $I$ which is homeomorphic to the closed interval $[0,1]$
and such that $I \setminus \{x,y\}$ is homeomorphic to the open interval $(0,1)$.
\item
$D$ has a single endpoint. 
Given a continuous, proper, and injective map $\gamma \colon \R_+ \to D$, there exists a constant $C >0$ such that 
$\gamma \big([C, +\infty)\big) = \{ x_r  \textrm{, }  1 - \varepsilon \le r < 1\}$ for some $0<\varepsilon <1$. 
\item
$A$ has two endpoints. 
Given a continuous, proper, and injective map $\gamma\colon \R_+ \to D$, there exist constants $C>0$ and $|t|<\varepsilon <1$ such that $\gamma \big([C, +\infty)\big)$ is either equal to $\{ x_r  \textrm{, }  1 - \varepsilon \le r < 1\}$, or to $\{ x_r\textrm{, } |t| < r\le |t| + \varepsilon\}$. 
\end{enumerate}
\end{thm}

Now let $U$ be a subset of a non-archimedean link $\NL(X,Z)$ and assume that  $U$ is an annulus.
Fix a $k((t))$-analytic annulus $A$ such that $U \cong \For(A)$, and denote by $\Sigma(U)$ the subset of $U$ corresponding to the subset $\{ x_r \textrm{ s.t. } |t| < r < 1\}$ of $A$.
It consists of points of type 2 or 3 of $A$.
The fact that the subset $\Sigma(U)$ of $U$ does not depend on the choice of a $k((t))$-analytic annulus $A$ such that  $U \cong \For(A)$ is a consequence of the following proposition.

\begin{prop}\label{proposition skeleton annulus}
The subset $\Sigma(U)$ of $U$ coincides with the set of points of $U$ which have no neighborhood isomorphic to a disc.
\end{prop}

\begin{proof}
	Any neighborhood of a point of $\Sigma(U)$ in $U$ has at least two endpoints, and thus can't be a disc.
	Conversely, let as above $A$ be a $k((t))$-analytic annulus such that $U \cong \For(A)$.
	The complement of $\Sigma(U)$ in $A$ is the union of the open balls $D(z,|z|)$, for $z$ ranging among the rigid points of $A$.
	If $D=D(z,|z|)$ is such a ball, $z$ is a defined over a finite extension $k((s))$ of $k((t))$, and the analytic function $s$ is globally defined on $D$.
	Having a rational point as center and a rational radius, $D$ can be seen as an open $k((s))$-analytic disc, and therefore $\For(D)$ is a disc.
\end{proof}

The following lemma, which will be used in section \ref{section_proof_main}, shows how discs and annuli appear as open subspaces of $\NL(\A^2_k,0)$.

\begin{lem}\label{lemma_disc_annulus}
Let $T$ and $t$ denote two coordinates for $\mathbb A^2_k$ at $0$.
Then:
\begin{enumerate}
\item $\NL(\A^2_k,0) \setminus V(t)$ is a disc;
\item $\NL(\A^2_k,0) \setminus V(tT)$ is an annulus.
\end{enumerate}
\end{lem}

\begin{proof}
Note that 
$\mathrm{c}_{\A^{2}_{k}}^{-1}(0) = \big\{x\in \A^{2,\mathrm{an}}_{k} \,\big|\, |T(x)|<1, |t(x)|<1 \big\}$.
Therefore we have
\begin{align*}
\NL(\A^{2}_{k},0) \setminus V(t)
& \cong \big\{ x \in \mathrm{c}_{\A^{2}_{k}}^{-1}(0) \,\big|\, |t(x)|=\varepsilon\big\} \\
& = \big\{x\in \A^{2,\mathrm{an}}_{k} \,\big|\, |T(x)|<1, |t(x)|=\varepsilon\big\} \\
& \cong \For\big(\big\{x\in \A^{1,\mathrm{an}}_{k((t))} \,\big|\, |T(x)|<1\big\}\big) = \For(D).
\end{align*}
The isomorphisms above respect the ringed space structure, proving $(i)$.

Set now $t'=tT$.
Then we have
\begin{align*}
\NL(\A^2_k,0)\setminus V(t')
& \cong \big\{\mathrm{c}_{\A^{2}_{k}}^{-1}(0) \,\big|\, |t(x)|=\varepsilon\big\}
 \\
& = \big\{x\in \A^{2,\mathrm{an}}_{k} \,\big|\, |T(x)|<1, |t(x)|<1, |t'(x)|=\varepsilon\big\}
 \\
& = \big\{x\in \A^{2,\mathrm{an}}_{k} \,\big|\, |t'(x)|<|T(x)|<1, |t'(x)|=\varepsilon\big\}
 \\
& \cong \For\big(\big\{x\in \A^{1,\mathrm{an}}_{k((t'))} \,\big|\, |t'|<|T(x)|<1\big\}\big)
 = \For(A),
\end{align*}
which concludes the proof of $(ii)$.
\end{proof}

\begin{rmk}\label{remark_disc_minus_point}
Note that $V(t)$ and $V(T)$ are single points of $\NL(\A^{2}_{k},0)$.
Observe that the lemma also shows that a punctured open $k((t))$-analytic disc $D\setminus\{0\}$ and an open $k((t))$-analytic annulus of modulus one, which are not isomorphic as analytic spaces, have isomorphic underlying ringed spaces in $k$-algebras.
\end{rmk}

\subsection{Morphisms}
A morphism $\NL(X,Z) \to \NL(X',Z')$ between two non-archimedean links is a morphism of the underlying ringed spaces
\[
\Big(\NL(X,Z),\mathcal O_{\NL(X,Z)},\mathcal O^\circ_{\NL(X,Z)}\Big) 
\longrightarrow 
\Big(\NL(X',Z'),\mathcal O_{\NL(X',Z')},\mathcal O^\circ_{\NL(X',Z')}\Big)
\]
which can be locally lifted to a morphism of $k((t))$-analytic spaces.
Rather than giving the precise definition from \cite[6.1]{fantini:normspaces}, we content ourselves with giving a list of example of morphism of non-archimedean links, including all the morphisms which will be considered in this paper.

\begin{enumerate}[label=(\roman{enumi})]

\item 
An \emph{isomorphism} of non-archimedean links $\NL(X,Z) \to \NL(X',Z')$ is an isomorphism of the underlying ringed spaces.

\item \label{item:bloqups_induce_isomorphism} 
As noted in Proposition~\ref{proposition_invariance_modifications}, if $ f \colon (X,Z) \to (X',Z')$ is a modification then the induced morphism $f \colon \NL(X,Z) \to \NL(X',Z')$ is an isomorphism. 

\item 
A morphism of $k$-varieties $f \colon X \to X'$ such that $f^{-1}(Z') = Z$ induces a morphism of non-archimedean links $f \colon \NL(X,Z) \to \NL(X',Z')$.

\item 
In particular, if $Y \subsetneq X$ is a subvariety such that $Y\cap Z$ is nowhere dense in $Y$, 
then the induced morphism $\NL(Y, Y\cap Z) \to \NL(X,Z)$ is a closed immersion, that is a map which lifts to closed immersions of $k((t))$-analytic spaces in the sense of Berkovich.
Its image is the closed subspace of $\NL(X,Z)$ consisting of the elements coming from elements of $Y^\mathrm{an}$; note that each such point is a semi-norm with nontrivial kernel.

\item 
If $(X,Z)$ are as above and $F$ is a closed subvariety of $Z$, then $\cent_X^{-1}(F)$, which is isomorphic to $\NL(X,F)\setminus Z^\mathrm{an}$, is an open subspace of $\NL(X,Z)$.

\end{enumerate}

\begin{rmk}
Observe that the converse to \ref{item:bloqups_induce_isomorphism} is partially true, in the following sense. 
By \cite[Theorem~6.4]{fantini:normspaces}, if $\NL(X,Z)$ is isomorphic to $\NL(X',Z')$ then there exist modifications $(Y,D)\to(X,Z)$ and $(Y',D')\to(X',Z')$ such that the corresponding formal completions $\widehat{Y/D}$ and $\widehat{Y'/D'}$ are isomorphic as formal schemes.
\end{rmk}



\section{Non-archimedean links of surfaces}
\label{section_links_surfaces}

In this section we assume that $X$ is a \emph{normal algebraic surface} over an algebraically closed field $k$ and that $0$ is a closed point of $X$.
Our approach follows the one of \cite[Sections 7, 9, 10]{fantini:normspaces}, but for some results we will provide more concrete proofs.

\subsection{Topology of $\NL(X,0)$ and center maps}
In dimension two the topology of $\NL(X,0)$ can be described in terms of its center maps as follows.

\begin{prop}\label{prop:topo_NL}
	Let $(X,0)$ be a normal surface singularity, and let $x$ be a point of $\NL(X,0)$.
	Then the family of all sets of the form $\mathrm c_Y^{-1}\big(\overline{\mathrm c_Y(x)}\big)$, for $(Y,D)$ ranging over all the modifications of $(X,0)$, is a basis of neighborhoods of $x$ in $\NL(X,0)$. 
\end{prop}

\begin{proof}
	Let $\mathfrak{m}$ be the maximal ideal of $\cO_{X,0}$, which consists of the functions vanishing at $0$, let $V$ be any affine neighborhood of $0$ in $X$, and denote by $\mathrm{L}(X,\mathfrak{m})$ the subset of $U^{\mathrm{an}}$ consisting of the multiplicative semi-norms $x$ on $\mathcal O_X(U)$ whose restriction to $k$ is trivial and such that 
	$\min_{f\in \mathfrak{m}} -\log|f(x)| = 1$.
	It is a compact subset of $\big( \mathrm{c}_X^{-1}(0)\setminus \{0\}^{\mathrm{an}} \big)$ and the natural projection $\mathrm{L}(X,\mathfrak{m})\to \NL(X,0)$,
	is bijective hence a homeomorphism. 
	It is then sufficient to prove that the family of all subsets of $\mathrm L(X,\mathfrak{m})$ of the form $\mathrm c_Y^{-1}\big(\overline{\mathrm c_Y(x)}\big)$, for $(Y,D)$ ranging over all the modifications of $(X,0)$, is a basis of neighborhoods of any given point $x$ of $\mathrm L(X,\mathfrak{m})$. 
	Any set of the form $\mathrm c_Y^{-1}\big(\overline{\mathrm c_Y(x)}\big)$ contains $x$, and it is open because the center map is anti-continuous. 
	We have to prove that, given any open subset $U$ of $\mathrm L(X,\mathfrak{m})$ containing $x$, there exists a modification $(Y,D) \to (X,0)$ such that $\mathrm c_Y^{-1}\big(\overline{\mathrm c_Y(x)}\big) \subset U$.
	Since $X^{\mathrm{an}}$ is endowed with the weakest topology making all evaluations maps $y \mapsto |f(y)|$ continuous, it is sufficient to prove it assuming that $U$ is a finite intersection of sets of the form
	$U_<(f,p,q)= \{y, |f(y)| <e^{-p/q}\}$ or $U_>= \{y, |f(y)| >e^{-p/q}\}$, where $f$ is an element of $\mathcal O_X(U)$ and $p,q$ are coprime integers.
	Moreover, since $|f| =1$ on $\mathrm{L}(X,\mathfrak{m})$ whenever $f\notin \mathfrak{m}$, we can also assume that $f\in \mathfrak{m}$.
	It is sufficient to find a modification $(Y,D) \to (X,0)$ such that  $U_<(f,p,q)$ (resp.  $U_>(f,p,q)$) is included in a set of the form $ c_Y^{-1} (\cE_<)$ (resp. $ c_Y^{-1} (\cE_>)$), for some closed subschemes $\cE_<$ and $\cE_>$ of $D$. 
	In order to do so we proceed as follows.
%
%
		Let $g_1, \ldots, g_N$ be a finite set of generators of $\mathfrak{m}$, and choose a modification $\pi\colon (Y,D) \to (X,0)$ 
		such that $g_i^p/f^q \circ \pi $ defines a regular map  $Y \to \mathbb{P}^1_k$ for all $i=1, \ldots, N$.
		Let $\cE_>$ be the union of all  (scheme-theoretic) points $\xi$ of $D$ such that 
		\begin{equation}\label{eq:>}
		\ord_\xi( f^q\circ \pi)> \ord_\xi(\mathfrak{m}^p) := \mathrm{min}_{g\in\mathfrak M^p}\{\ord_\xi(g\circ\pi)\}
		= \mathrm{min}_{g\in\mathfrak M}\{\ord_\xi(g^p\circ\pi)\}~.
		\end{equation}
		Observe that $z\in \cE_>$ if and only if the value of the rational function $g_i^p/f^q (\pi(z))$ is equal to  $\infty \in \mathbb P^1_k$ for at least one index $i$, which implies that $\cE_>$ is a finite union of closed points in $D$ together with those irreducible components $E$ of $D$ whose generic point $\xi_E$ satisfies~\eqref{eq:>}.
		In particular $\cE_>$ is closed in $D$. One defines in the same way $\cE_<$ as the set of points $\xi\in D$ satisfying
		\[
		\ord_E( f^q\circ \pi)< \ord_E(\mathfrak{m}^p),
		\]
		and similarly $\cE_<$ is closed (it is in fact a finite union of irreducible components of $D$). 
		
		Since $g_i^p/f^q \circ \pi $ is regular, we have that $\cE_> \cap \cE_< =\emptyset$, and moreover
	\begin{align*}
		U_<(f,p,q) & = \{y, \, -\log |f^q(y)| > p\} = c_Y^{-1} (\cE_<)
		\\
		U_>(f,p,q) & = \{y,\, -\log |f^q(y)| < p\} = c_Y^{-1} (\cE_>)~,
	\end{align*}
	which concludes the proof. 
\end{proof}

\subsection{Types of points}
Observe that a point $x$ of $\NL(X,0)$ corresponds to an equivalence class of semi-norms on $X$; being centered in $0$, those induce semi-norms on the completed local ring $\widehat {\mathcal O_{X,0}}$.
The points of $\NL(X,0)$ can then be divided into four different types by looking at the associated valued field ${\mathscr H(x)}$, its trascendence degree over $k$, and its valuative invariants.

We say that a point $x$ of $\NL(X,0)$ is a \emph{rigid point} (or of \emph{type 1}) if the transcendence degree $\mathrm{trdeg}_{k}{\mathscr H(x)}$ of $\mathscr H(x)$ over $k$ is equal to $1$.
Equivalently, $x$ is a rigid point if it corresponds to an equivalence class of semi-norms on $\widehat {\mathcal O_{X,0}}$ with nontrivial kernel.
Moreover, this kernel is generated by an irreducible element of $\widehat {\mathcal O_{X,0}}$, so that $x$ corresponds to an irreducible germ of a formal curve on $(X,0)$; then $x$ can be seen as the equivalence class of the order of vanishing along this germ.
When this is the case, the rational rank
\[
\mathrm{rank}_\Q \big( |\mathscr H(x)^\times|/|k^\times|\otimes_\Z\Q \big)
\]
of $\mathscr H(x)$ is equal to $1$ and its residue field $\widetilde{\mathscr H(x)}$ is equal to $k$.

In every other case we have $\mathrm{trdeg}_{k}{\mathscr H(x)}=2$.
We say that $x$ is a \emph{divisorial point} (or of \emph{type 2}) if the residue field $\widetilde{\mathscr H(x)}$ is a nontrivial extension of $k$, while we say that $x$ is of \emph{type 3} if the rational rank of $\mathscr H(x)$ is equal to $2$.
Finally, we say that $x$ is of \emph{type 4} if it satisfies $\mathrm{rank}_\Q \big( |\mathscr H(x)^\times|/|k^\times|\otimes_\Z\Q \big)=1$ and $\widetilde{\mathscr H(x)}=k$ but it is not of type 1 (that is, the corresponding semi-norms on $\widehat{\cO_{X,0}}$ are norms).

Every point of $\NL(X,0)$ is of one (and only one) of the four types above, since by Abhyankar's inequality \cite{abhyankar:valuationscentered} we have
\[
\mathrm{rank}_\Q \big( |\mathscr H(x)^\times|/|k^\times|\otimes_\Z\Q \big)+\mathrm{trdeg}_{k}\widetilde{\mathscr H(x)}\leq \mathrm{trdeg}_{k}{\mathscr H(x)} \leq 2,
\]
and $\mathrm{rank}_\Q \big( |\mathscr H(x)^\times|/|k^\times|\otimes_\Z\Q \big)\geq1$ because the valuation associated with $x$ is nontrivial.

Observe that an isomorphism of non-archimedean links respects the complete residue fields, and therefore must send a point to a point of the same type.

\begin{rmk}
Recall that $\NL(X,0)$ is locally isomorphic to a $k((t))$-analytic curve. Under any such isomorphism the type of a 
point as defined above coincides with its type as defined by Berkovich. 
For a definition of types of points in Berkovich curves in terms of their valuative invariants see e.g. \cite[3.3.2]{ducros:structurecourbesanalytiques}.
\end{rmk}

\begin{rmk}
Let $(Y,D)$ be a modification of $(X,0)$ and let $x$ be a point of $\NL(X,0)$.
The existence of a morphism $\overline\alpha\colon \Spec \big( \mathscr H(x)^\circ \big) \to Y$ associated with $x$ implies that the residue field $\widetilde{\mathscr H(x)}$ of $\mathscr H(x)$ is a field extension of the schematic residue field of $Y$ at $\mathrm{c}_Y(x)$.
It follows that $\cent_Y(x)$ is a closed point of $Y$ if $x$ is not divisorial.
Proposition~\ref{prop:topo_NL} implies then that such a point $x$ has a basis of neighborhoods of the form $\cent_Y^{-1}\big({\mathrm c_Y(x)}\big)$, for $(Y,D)$ ranging over all the modifications of $(X,0)$. 
\end{rmk}

\begin{rmk}\label{rmk:divisorial}
	Let $(Y,D)$ be a modification of $(X,0)$ and let $x$ be a point of $\NL(X,0)$.
	If $x$ is divisorial, then $\mathrm{c}_Y(x)$ is either a closed point of $Y$ or the generic point of an irreducible curve in $Y$.
	Moreover, it is always possible to find a modification $(Y',D')$ of $(X,0)$ that dominates $(Y,D)$ and such that $\mathrm{c}_{Y'}(x)$ is the generic point of an irreducible curve $E$ in $Y'$, which explains our terminology. 
	Furthermore, the residue field $\widetilde{\mathscr H(x)}$ of ${\mathscr H(x)}$ is then isomorphic to the function field $k(E)$ of $E$.
\end{rmk}

\subsection{Formal modifications and fibers of the center maps}
Let $X$ be a normal $k$-surface and let $0$ be a closed point of $X$.
We start by fixing some notation.

A modification $(Y,D)$ of $(X,0)$ is said to be a \emph{resolution} of $(X,0)$ if $Y$ is regular.
If moreover the irreducible components of $D$ are all non-singular, intersect transversally and no three of them meet at a point, then $(Y,D)$ is said to be a \emph{good resolution} of $(X,0)$. 
Whenever $C$ is a germ of (formal) curve in $(X,0)$, then a good resolution $(Y,D)$ of $(X,0)$ is also said to be a good resolution of $C$ if the strict transform of $C$ in $Y$ meets $D$ transversally. 


We mentioned that the non-archimedean link of a pair $(Y,D)$ only depends on an infinitesimal neighborhood of $D$ in $Y$.
The notions above can then be slightly generalized by working in a suitable category of formal $k$-schemes.
A formal $k$-scheme $\Y$ is called \emph{special} if it is covered by formal subschemes $\Y_i$ such that each $\Y_i$ can be written as the formal completion of a $k$-scheme of finite type $Y_i$ along a closed subscheme $D_i$.
We can then define the non-archimedean link $\NL(\Y)$ of $\Y$ by gluing the non-archimedean links $\NL(Y_i,D_i)$.
Observe that when $\Y=\widehat{Y/D}$ is the formal completion of $Y$ along $D$, then $\NL(\Y)=\NL(Y,D)\cong\NL(X,0)$.
We also obtain a center map $\cent_\Y\colon \NL(\Y) \to \Y_0$, where the reduction $\Y_0$ of $\Y$ is a scheme of finite type over $k$ covered by the reduced schemes associated to the $D_i$.

A special formal $k$-scheme $\Y$ is called a \emph{formal modification} of the pair $(X,0)$ if it is normal and it comes endowed with an adic morphism $f\colon\Y\to\X=\widehat{X/0}$ that induces an isomorphism of non-archimedean links $\NL(\Y)\stackrel{\sim}{\longrightarrow}\NL(X,0)$ and such that the fiber product $\Y\times_\X \{0\}$ is a Cartier divisor of $\Y$.
If $(Y,D)\to(X,0)$ is a modification, then the formal completion $\Y=\widehat{Y/D}\to\X$ of $Y$ along $D$ is a formal modification of $(X,0)$.
Such a formal modification $\Y$ of $(X,0)$ is said to be \emph{algebraizable}, and a modification $(Y,D)$ of $(X,0)$ such that $\Y\to\X$ is isomorphic to $\widehat{Y/D}\to\X$ is called an \emph{algebraization} of $\Y$.
If $\Y$ is a formal modification of $(X,0)$ such that $\Y$ is regular, then by \cite[Proposition~7.6]{fantini:normspaces} $\Y$ is algebraized by a resolution of $(X,0)$.
Observe that a formal modification $\Y$ of $(X,0)$ induces an isomorphism of non-archimedean links $\NL(\Y)\cong\NL(X,0)$, and therefore also a center map $\cent_\Y\colon \NL(X,0) \to \Y_0$.

If $\Y$ is a formal  modification of $(X,0)$, we denote by $\Div(\Y)$ the finite nonempty subset of $\NL(X,0)$ consisting of the divisorial points associated with the components of $\Y_0$. 
Whenever $\Y$ is algebraized by a modification $(Y,D)$, we will also denote $\Div(\Y)$ by $\Div(Y)$.

The following proposition is a simple consequence of Lemmas 7.14 and 9.3 of \cite{fantini:normspaces}.

\begin{prop}\label{proposition_propertiesNL}
Let $(X,0)$ be a normal surface singularity and let $\Y$ be a formal modification of $(X,0)$.
Then the following properties hold:
\begin{enumerate}[label=(\roman{enumi}),ref=\thethm.({\roman*})]
\item \label{proposition_propertiesNL_2}
 The map $\cent_{\Y}^{-1}$ gives a bijection between the set of closed points of $\Y_0$ and the set of connected components of $\NL(X,0)\setminus \Div(\Y)$.
\item \label{proposition_propertiesNL_3}
Let $W$ be a connected component of $\NL(X,0)\setminus \Div(\Y)$, let $p$ be the corresponding closed point of $\Y_0$, let $\Y_p=\Spf \big(\widehat{\cO_{\Y,p}} \big)$ be the formal completion of $\Y$ along $p$, let $\mc I_p$ be the ideal of $\widehat{\cO_{\Y,p}}$ which defines $\Y_0$ locally around $p$, and denote by $\varphi\colon\NL(\Y_p)\to \NL(X,0)$ the map induced by the composition $\Y_p\to\Y\to\X$.
Then $\varphi$ maps the zero locus $V(\mc I_p)$ of $\mc I_p$ in $\NL(\Y,p)$ to a finite set of type 1 points of $\NL(X,0)$, and it induces an isomorphism $\NL(\Y_p)\setminus V(\mc I_p) \cong W \subset \NL(X,0)$.
\end{enumerate}
\end{prop}

\begin{rmk}\label{remark_closure_component}
Let $\Y$ and $W$ be as above.
It follows from the first part of the proposition that the closure $\overline W$ of $W$ in $\NL(X,0)$ is obtained by adding to $W$ a subset of $\Div(\Y)$.
Indeed, the union of all the connected components of $\NL(X,0)\setminus \Div(\Y)$ different from $W$, that is the complement of $W\cup \Div(Y)$ in $\NL(X,0)$, is an open subset of $\NL(X,0)$.
\end{rmk}


\subsection{Existence of formal modifications and resolutions}
\label{subsection_existencemodifications}
We will now explain how to produce formal modifications of $(X,0)$ with prescribed exceptional divisors and how to detect when such a modification is a good resolution of $(X,0)$.

\begin{thm}\label{theorem_existence_modifications}
Let $S$ be a finite subset of divisorial points of $\NL(X,0)$.
Then there exists a formal modification $\Y$ of $(X,0)$ such that $\Div(\Y)=S$.
\end{thm}
\begin{proof}
As it readily follows from the observation made in Remark~\ref{rmk:divisorial}, there exists a resolution $(Y,D)$ of $(X,0)$ such that $S\subset \Div(Y)$. 
The contractibility criterion of Grauert-Artin \cite{artin:contractibilityalgebraicspaces} guarantees that we can contract every component of $D$ which does not correspond to an element of $S$, yielding a normal algebraic space over $k$ with a proper morphism $f$ to $X$.
Indeed, the intersection matrix of the divisor that we want to contract is negative definite because the entire exceptional divisor of $Y$ can be contracted to the point $0$ in $X$.
By taking the formal completion of this algebraic space along $f^{-1}(0)$ we obtain the formal modification $\Y$ that we wanted.
\end{proof}

\begin{rmk}
If we are working over the field of complex numbers, we can apply Grauert contractibility criterion \cite{grauert:ubermodifikationen} instead of Artin's and obtain $\Y$ as a complex analytic space.
Of course this is the same as analytifying the algebraic space given by Artin's criterion.
On the other hand, observe that if $k$ is the algebraic closure of $\mathbb F_p$ for some prime $p$ or if $(X,0)$ is a rational singularity, then the contractibility results \cite[2.3, 2.9]{artin:contractibilitycriterion} grant that $\Y$ is an algebraic variety.
\end{rmk}

Recall that each type 1 point can be seen as the equivalence class of the order of vanishing along an irreducible germ of a formal curve on $(X,0)$.
In particular, with a finite set of type 1 points of $\NL(X,0)$ is associated the germ of a curve on $(X,0)$.

\begin{thm}\label{theorem_characterization_resolutions}
Let $\Y$ be a formal modification of $(X,0)$, let $T$ be a finite set of type 1 points of $\NL(X,0)$ and let $C$ be the germ of curve on $(X,0)$ associated with $T$.
Then $\Y$ can be algebraized by a good resolution of both $(X,0)$ and the germ $C$ if and only if each connected component $V$ of $\NL(X,0)\setminus \Div(\Y)$ has one of the following three forms:
\begin{enumerate}
\item $V$ is a disc and $V\cap T=\emptyset$; 
\item $V$ is an annulus and $V\cap T=\emptyset$; 
\item $V$ is a disc and $V\cap T$ can be taken to be its origin.
\end{enumerate}
\end{thm}

Observe that in the statement can take $T$ to be empty, so that there is no curve $C$ and we simply obtain a good resolution of $(X,0)$.
The following proof follows the lines of the proof of Proposition 10.2 of \cite{fantini:normspaces}, where more details are given.

\begin{proof}[Proof]
By Proposition~\ref{proposition_propertiesNL_2} the connected components of $\NL(X,0)\setminus \Div(X)$ are the inverse images through the center map of the closed points of $\Y_0$.
Let $W$ be such a component, let $p$ be a closed point of $\Y_0$, and let $\mathcal I_p$ be the image of the ideal defining $\Y_0$ in $\mathcal O_{\Y,p}$.
By Cohen theorem, $\Y$ is regular at $p$ if and only if $\widehat{\mathcal O_{\Y,p}}\cong k[[x,y]]$, that is if and only if $\NL(\Y_p)\cong \NL(\mathbb A^2_k,0)$.
Moreover, $p$ is a smooth point of $\Y_0$ if and only we can take $\mathcal I_p=(x)$ in the isomorphism above, while $p$ is an ordinary double point of $\Y_0$ if and only if we can take $\mathcal I_p=(xy)$.
Since by Proposition~\ref{proposition_propertiesNL_3} we have a canonical isomorphism $\cent_\Y^{-1}(p)\cong\NL({\Y_p})\setminus V(\mathcal I_p)$, it follows from Proposition~\ref{proposition_propertiesNL_1} that $\Y$ is a (formal) good modification of $(X,0)$ if and only if every connected component of $\NL(X,0)\setminus \Div(\Y)$ is either a disc or an annulus. 
Whenever this is the case, $\Y$ is algebraized by a good resolution $(Y,D)$ of $(X,0)$ by \cite[Proposition 7.6]{fantini:normspaces}.
Finally, by definition $(Y,D)$ is also a good resolution of the germ $C$ if and only if the strict transform of $C$ in $Y$ meets $D$ transversally. 
This means that if $p$ is a point of $D$ contained in the strict transform of $C$ we can find an isomorphism $\widehat{\mathcal O_{Y,p}} \cong k[[x,y]]$ as above with $\mathcal I_p=(x)$, and the strict transform locally defined by $y$ at $p$.
This corresponds precisely to the conditions on $T$ in the statement.
\end{proof}

\subsection{Structure of $\NL(X,0)$}
The structure of the non-archimedean link $\NL(X,0)$ can be described using a resolution of the singularities of the pair $(X,0)$ and the results of the previous section.

\begin{cor}\label{corollary_global_topology}
	Let  $(Y,D)$ be a good resolution of $(X,0)$ and let $y$ be a closed point of $D$.
	If $D$ is smooth (respectively singular) at $y$   then $c_Y^{-1}(y)$ is a disc (resp. an annulus), and the boundary 
	of  $c_Y^{-1}(y)$ in $\NL(X,0)$ consists of one type 2 point (resp. two type 2 points). 
	In particular, $\NL(X,0) \setminus  \Div(Y)$ is a disjoint union of discs and finitely many annuli. 
\end{cor}

\begin{proof}
	As above, the space $Y$ is smooth at $y$, so by Cohen theorem we can find an isomorphism $\widehat{\mathcal O_{Y,y}} \cong k[[U,V]] \cong \widehat{\mathcal O_{\mathbb A^2_k,0}}$ such that a local equation for $D$ at $y$ is either $U$ (if $y$ is a smooth point of $D$) or $UV$ (if $y$ is a double point of $D$).
	Since $\NL\big(\Y_{y}\big)\cong\NL\big(\mathbb A^2_k,0\big)$, Proposition~\ref{proposition_propertiesNL_3} and Lemma~\ref{lemma_disc_annulus} imply that $\cent_Y^{-1}\big(y \big)$ is isomorphic to an open $k((t))$-analytic disc in the first case and an annulus in the second case, where the $k((t))$-analytic structure is defined by sending $t$ to a local equation of $D$ at $y$.
	The fact that the boundary $\cent_Y^{-1}(y)$ consists of one or two points follows from Remark~\ref{remark_closure_component}. 
	The last statement is now a consequence of Proposition~\ref{proposition_propertiesNL_2}.
\end{proof}

We can now deduce some results about the type 1 points of $\NL(X,0)$.

\begin{cor}\label{corollary_topology}
	Any type 1 point $x$ of $\NL(X,0)$ admits a basis of neighborhoods consisting of discs centered in $x$.
	Moreover, the set of type 1 points is dense in $\NL(X,0)$.
\end{cor}

\begin{proof}
	By Proposition~\ref{prop:topo_NL} a basis of neighborhoods of $x$ consists of the open subsets of $\NL(X,0)$ of the form $\big\{\cent_Y^{-1}\big({\overline{\cent_Y(x)}}\big)\big\}$, for $(Y,D)$ ranging among the good resolutions of $(X,0)$, since the family of good resolutions is cofinal among the partially ordered set of modifications of $(X,0)$.
	Since $x$ is not of type $2$, $\cent_Y(x)$ is a closed point of $D$, and we can find an isomorphism $\widehat{\mathcal O_{Y,y}} \cong k[[U,V]] \cong \widehat{\mathcal O_{\mathbb A^2_k,0}}$ such that $U$ is a local equation for $D$ at $y$ and $V$ defines the germ of curve at $y$ associated to the type one point $x$.
	This shows that $\cent_Y^{-1}\big(\overline{\cent_Y(x)}\big)=\cent_Y^{-1}\big(\cent_Y(x)\big)$ is an open $k((t))$-analytic disc centered in $x$.
	
	It remains to prove the density of the set of type 1 points of $\NL(X,0)$.
	As noted in Remark~\ref{remark_closure_component}, each of the divisorial points associated with a good resolution of $(X,0)$ is not isolated in $\NL(X,0)$.
	The result then follows from Corollary~\ref{corollary_global_topology}, Lemma~\ref{lemma_disc_annulus}, and the fact that the set of its points of type 1 is dense in a $k((t))$-analytic annulus $A=\big\{x\in \A^{1,\mathrm{an}}_{k((t))} \,\big|\, |t|_{k((t))}<|T(x)|<1\big\}$. 
	The latter is a classical fact that can be proven directly by exhibiting suitable type 1 points; as an example, the semi-norm that sends an element $P$ of $k((t))[T]$ to $|P(x)| = |P(t^{1/2})|_{k((t))}$ is a type 1 point $x$ of $A$.
\end{proof}

\begin{rmk}
	It is also true that each of the sets consisting of type 2, type 3 or type 4 points is dense in $\NL(X,0)$, but we will not need this fact. 
\end{rmk}

Any bounded analytic function on the complement of finitely many type 1 points extends to $\NL(X,0)$. 
This follows from the fact that $\NL(X,0)$ is locally a $k((t))$-analytic space combined with the more general result \cite[Proposition 3.3.14]{berkovich:book}.
However, we include a proof here since it is simpler to deduce the result in our very special case from Corollary~\ref{corollary_topology}.

\begin{lem}\label{lemma_extension_to_type1}
	Let $U$ be an open subset of $\NL(X,0)$ and let $S$ be a finite subset of $U$ consisting of type $1$ points.
	Then the inclusion $U\setminus S\hookrightarrow U$ induces an isomorphism
	\[
	\mathcal O_{\NL(X,0)}^\circ \big( U \big) \cong \mathcal O_{\NL(X,0)}^\circ \big( U\setminus S \big).
	\]
\end{lem}

\begin{proof}
	Since $\mathcal O^\circ_{NL(X,0)}$ is a sheaf it is enough to prove the result for $T=\{x\}$ a single type $1$ point, and $U$ some neighborhood of $x$ in $\NL(X,0)$.
	Hence, by Corollary~\ref{corollary_topology} we can assume without loss of generality that $U$ is a disc centered in $x$.
	Then the ring of bounded analytic functions on $U\setminus\{x\}$ is isomorphic to $k((t))[[T]]$, where $T$ is a coordinate function on $U$, and all such functions extend to $x=V(T)$.
\end{proof}

The results of this section can also be used to deduce a characterization of the \emph{essential valuations} of a surface singularities, as appearing in the Nash problem (see \cite[Theorems 10.4 and 10.8]{fantini:normspaces}), and to give another proof of the existence of resolutions for surfaces (see \cite[Theorem 8.6]{FantiniTurchetti2017}).

\section{Self-similarity of sandwiched singularities}\label{section_preliminariessandwiched}

In this section we introduce sandwiched singularities and prove the implications \ref{condition_sandwiched} $\implies$ \ref{condition_strongly_selfsim} $\implies$ \ref{condition_valtree}   $\implies$ \ref{condition_selfsim} of Theorem~\ref{mainthm}.

\subsection{Sandwiched singularities}\label{ssec:sandwich}

\begin{defi}\label{def-sandwich}
Let $\mathcal O$ be a normal 2-dimensional complete local ring with algebraically closed residue field $k$.
We say that $\mathcal O$ is \emph{sandwiched} if there exist two algebraic surfaces $X_0$ and $Y$ over $k$, with $X_0$ smooth over $k$ and $Y$ normal, a proper birational morphism $Y\to X_0$, and a point $y$ in $Y$ such that $\mathcal O\cong\widehat{\mathcal O_{Y,y}}$.
If $X$ is an algebraic surface over $k$ and $0$ is a normal point of $X$, we say that $(X,0)$ is a sandwiched singularity if the complete local ring $\widehat{\mathcal O_{X,0}}$ of $X$ in $0$ is sandwiched.
\end{defi}

Note that when working over $\C$ our definition is equivalent to \cite[Definition II.1.1]{spivakovsky:sandsingdesingsurfNashtransf} thanks to Artin's approximation theorem \cite[Corollary 1.6]{artin:solaneq}. 

\medskip

Fix  any smooth algebraic surface  $X_0$ over $k$, such as for example $X_0 = \A^2_k$. 
Let $p$ be a point of $X_0$, let $\varphi\colon Y\to X_0$ be a sequence of point blowups centered above $p$, and let $E$ be a connected divisor on $Y$ obtained by removing some irreducible components from $\varphi^{-1}(p)$.
Since the divisor $\varphi^{-1}(p)$ can be contracted to $(X_0,p)$ and $X_0$ is smooth, Artin's contractibility criterion \cite[Theorem 2.3]{artin:contractibilitycriterion} ensures that the divisor $E$ can be contracted to a point $0$ in a normal algebraic variety $X$. 
The normal singularity $(X,0)$ is sandwiched, and any sandwiched surface singularities is isomorphic \'etale locally (or analytically, if $k=\C$) to such a singularity. 

\begin{rmk}\label{rmk:sandwichoverC}
Suppose $(X,0)$ is sandwiched. 
Then in the category of analytic spaces over $k$, or in the complex analytic category when working over $\mathbb C$, one can always find a morphism $X\to X_0$ to a smooth surface. 
If $\mu\colon Y\to X$ is a resolution of $(X,0)$, we get back $(X,0)$ by contracting the divisor $\mu^{-1}(0)$ of $Y$.
This explains the terminology, as $X$ is sandwiched between the smooth surfaces $X_0$ and $Y$.
%
\end{rmk}

We start by proving a stronger form of the implication \ref{condition_sandwiched} $\implies$ \ref{condition_valtree}.

\begin{lem}\label{lem:structure sandw}
Suppose $(X,p)$ is a sandwiched singularity and $y$ is a point of $\NL(\A^2_k,0)$.
Then there exists  a finite set $T$ of type 1 points in $\NL(X,p)$ such that $\NL(X,p)\setminus T$ is isomorphic 
to an open subset of $\NL(\A^2_k,0)$ containing $y$. 
\end{lem}

\begin{proof}
We may assume that $\cO_{X,p}$ is not a regular ring.

Recall that $\NL(X,p)$ only depends on the formal completion of its local ring. We may thus assume 
that there exist a proper birational map from a smooth algebraic variety $Y$ to the affine plane $\varphi \colon Y \to \A^2_k$ which is an isomorphism outside $\varphi^{-1}(0)$,
and a connected divisor $E \subset \varphi^{-1}(0)$ such that $X$ is the surface obtained from $Y$ by contracting $E$ to a point. 
We denote by $\mu \colon Y \to X$ the contraction map
so that $p= \mu(E)$. 
Note that $\varphi$ factors through $\mu$, so there exists a regular birational map
$\pi\colon X \to \A^2_k$ mapping $p$ to $0$. 

It follows from \cite[Corollary 1.14]{spivakovsky:sandsingdesingsurfNashtransf} that we may further impose that any
irreducible component of $\varphi^{-1}(0)$ that is not included in $E$ has self-intersection $-1$. 
Since $Y$ is regular, $\varphi$ is a sequence of point blow-ups hence factors through the blow-up of the origin $\varpi\colon Y_1 \to \A^2_k$.  
The map $Y \to Y_1$ is not an isomorphism, therefore the strict transform $E_0$ of $\varpi^{-1}(0)$ in $Y$ has self-intersection at most $-2$, hence is included in $E$.

Let $T$ be the zero locus in $\NL(X,0)$ of the ideal defining $\pi^{-1}(0)$ around $p$ and let $\Y$ be the formal completion of $X$ along $\pi^{-1}(0)$.
By Proposition~\ref{proposition_propertiesNL_3} $T$ is a finite set of type 1 points of $\NL(X,0)$, and $\pi$ induces an isomorphism between $\NL(X,p) \setminus T$ and the open subset $U$ of $\NL(\A^2_k,0)$ consisting
of the points whose center in $X$ is equal to $p$. 
Note that, by the commutativity of the center maps, $U$ is also equal to the set of points of $NL(X,p)$ whose center on $Y$ is contained in $E=\mu^{-1}(p)$. 

If $U$ contains $y$ there is nothing left to prove.
If this is not the case, observe that the linear group $\GL(2,k)$ acts naturally on $\NL(\A^2_k,0)$ in such a way that any element $g\in \GL(2,k)$ defines an isomorphism of non-archimedean links $g_\bullet \colon \NL(\A^2_k,0) \to \NL(\A^2_k,0) $. 
The linear map $g$ also lifts to an automorphism $L_g \colon Y_1\to Y_1 $ whose action on the exceptional divisor $\varpi^{-1}(0)$ is given by the projectivization of $g$, and satisfies $\mathrm{c}_{Y_1} (g_\bullet (y)) = L_g( \mathrm{c}_{Y_1}(y))$. 
We may thus find $g\in \GL(2,k)$ such that  $ L_g( \mathrm{c}_{Y_1}(y))$, and hence $\mathrm{c}_{Y_1} (g_\bullet (y))$, does not belong to the indeterminacy locus of
the birational map $  \varphi^{-1} \circ \varpi \colon Y_1 \to Y$.
In particular, the single irreducible component of $\varphi^{-1}(0)$ containing the point $\mathrm{c}_{Y} (g_\bullet (y))$  is $E_0$, so that  $\mathrm{c}_{Y}(g_\bullet (y))\in E$, and $g_\bullet (y)$ belongs to $U$. 

It follows that the open subset $g_\bullet^{-1} (U)$ in $\NL(\A^2_k,0)$, which is isomorphic to $\NL(X,p) \setminus T$ because $U$ is and $g_\bullet$ is an isomorphism, contains $y$ as required.
This concludes the proof of the Lemma.
\end{proof}

\subsection{The implication \ref{condition_sandwiched} $\implies$ \ref{condition_strongly_selfsim}}
Pick any point $x\in \NL(X,0)$ which is not of type 2 and an open set $U$ containing $x$. 
By Proposition~\ref{prop:topo_NL} one may choose a good resolution $Y \to X$ and a point $p$ on its exceptional divisor
such that $x\in \mathrm{c}_Y^{-1}(p)\subset U$.

Proposition~\ref{proposition_propertiesNL} (ii) implies that the good resolution $Y \to X$ induces an isomorphism $\mathrm{c}_Y^{-1}(p) \cong \NL(\Y_p) \setminus V(\mathcal{J}_p)$. Observe that the latter non-archimedean link is isomorphic to $\NL(\A^2_k,0) \setminus S$ where $S$ is a set of type 1 points of cardinality $1$ or $2$ depending on whether $p$ is a smooth point of the exceptional divisor or not. Denote by $y$ the image of $x$ under this isomorphism.

By Lemma~\ref{lem:structure sandw} there exists a finite set $T$ of type 1 points and an isomorphism of non-archimedean links $\psi$ between  $\NL(X,0)\setminus T$  and  an open subset of $\NL(\A^2_k,0)$ containing $y$. 
Adding $\psi^{-1}(S)$ to $T$ if necessary we may suppose that $\psi(\NL(X,0)\setminus T) \subset  \NL(\A^2_k,0)\setminus S$. 

We conclude observing that $\varphi^{-1} \circ \psi$ is an isomorphism mapping $\NL(X,0)\setminus T$ to $\mathrm{c}_Y^{-1}(p)$ and containing $x$. \hfill$\qed$

\subsection{The implication  \ref{condition_strongly_selfsim} $\implies$ \ref{condition_valtree}}
By Corollary~\ref{corollary_topology}, one can find a disc $D$ in $\NL(X,0)$. By~\ref{condition_strongly_selfsim}, there exists a finite set $T$ of type 1 points and an open subset of $D$
which is isomorphic to $\NL(X,0)\setminus T$. By Lemma~\ref{lemma_disc_annulus}, $D$ can be realized as an open subset of $\NL(\A^2_k,0)$.
\hfill$\qed$

\subsection{The implication \ref{condition_valtree} $\implies$ \ref{condition_selfsim}}
Let $V$ be any open subset of $\NL(X,0)$.
By Corollary~\ref{corollary_topology}, $V$ contains a point of type 1, and therefore it contains a disc $D$ by the same corollary.
Now fix a finite set $T$ of type 1 points in $\NL(X,0)$ such that condition \ref{condition_valtree} holds. 
Adding one more point to $T$ if necessary, one may assume that 
$\NL(X,0)\setminus T$ is isomorphic to an open subset of $\NL(\A^2_k,0)\setminus V(t)$, which is a disc by Lemma~\ref{lemma_disc_annulus}.
We conclude that $\NL(X,0)\setminus T$  can be realized as an open subset of $D\subset V$. 
\hfill$\qed$



\section{Self-similar graphs}\label{section_graphs}
In this section we introduce the notions of self-similar and sandwiched graphs
and prove that these notions are in fact equivalent. 

\subsection{Modifications of graphs}
Let us introduce some terminology first. A \emph{graph} $\Gamma$ is the data of a finite set of vertices $V= V(\Gamma)$ and a subset $E= E(\Gamma)$ of pairs in $V$
(the set of edges). In particular our graph has no loop. Two vertices are said to be connected (or joined) by an edge when $\edge{v_1}{v_2}\in E$.
In that case, we write $v_1 \sim v_2$.
A graph is connected  if  for any two vertices $v_0, v_1$ there exists a sequence of edges
$\edge{v_0}{w_1}, \edge{w_1}{w_2}, \ldots, \edge{w_n}{v_1}$ joining them.  A \emph{weighted graph} is a graph together with a function 
$V\to \N\times \N^*$. We shall write this function as $v\mapsto (g(v),e(v))$. 

The \emph{link} of a vertex $v$ of a graph $\Gamma$ is by definition the set $L(v)$ of vertices $w$ such that $\edge{v}{w}$ is an edge. 
The cardinality of $L(v)$ is the \emph{valency} of $v$.

\begin{rmk}
In order to motivate our further  definitions, let us indicate in which situation graphs will arise in the sequel. 
We shall consider the dual graph of a resolution of a normal surface singularity, 
so that vertices are in bijection with exceptional components of this resolution. The weight function is then given 
by genus  and the opposite of the self-intersection of a component.
\end{rmk}


If $\Gamma$ is a graph, a \emph{simple modification of} $\Gamma$ \emph{centered  at a vertex} $v$ is a graph $\Gamma'$ such that 
its set of vertices is $V'= V \cup\{v'\}$, 
its set of edges is $E' = E \cup \big\{ \edge{v}{v'} \big\}$, 
and its weight function $(g',e')$ satisfies $g'(w) = g(w)$ and $e'(w) = e(w)$ for all $w\in V\setminus\{v\}$, $g'(v) = g(v)$, $g'(v') =0$, $e'(v) = e(v) +1$, and $e(v')=1$. 

A \emph{simple modification of} $\Gamma$ \emph{centered at an edge} $\edge{v_0}{v_1}\in E$ is a graph $\Gamma'$ such that 
its set of vertices is $V'= V \cup\{v'\}$, 
its set of edges is $ E' = \big(E \setminus \big\{ \edge{v_0}{v_1}\big\}\big) \cup \big\{\edge{v_0}{v'}, \edge{v_1}{v'}\big\}$, 
and its weight function $(g',e')$ satisfies $g'(w) = g(w)$ for all $w\in V$, $g(v') =0$, $e'(w) = e(w)$ for all $w\in V \setminus \{v_0,v_1\}$, $e'(v_0) = e(v_0) +1$, $e'(v_1) = e(v_1) +1$, and $e(v')=1$. 

\smallskip

A \emph{modification} of a graph $\Gamma$  is a graph $\Gamma'$ which is obtained from $\Gamma$ by a finite sequence of simple modifications (centered either at a vertex or at an edge). 
A modification is nontrivial whenever it is not an isomorphism. 
We shall write $\Gamma' \rightsquigarrow \Gamma$ to say that $\Gamma'$ is a modification of $\Gamma$.
Note the following transitivity property: if $\Gamma''\rightsquigarrow \Gamma'$ and $\Gamma' \rightsquigarrow \Gamma$ are modifications, then
$\Gamma'' \rightsquigarrow \Gamma$ is also a modification.

Observe that a modification $\Gamma' \rightsquigarrow \Gamma$ yields a canonical inclusion $\imath\colon V(\Gamma) \to V(\Gamma')$ of the set of vertices of $\Gamma$ into $\Gamma'$. 
The image of a vertex $v$ by $\imath$ is called its \emph{strict transform}. Note that in general two vertices in $V$ joined by an edge need not have strict transforms joined by an edge in $\Gamma'$.

We can also define the \emph{total transform} of a connected subgraph $\Delta$ of $\Gamma$ by a modification $\Gamma' \rightsquigarrow \Gamma$. 
It is sufficient to explain the construction of the total transform in the case of a simple modification. 
If the center of the modification is a vertex not belonging to $\Delta$ or an edge $\edge{v}{w}$ with $v,w\notin \Delta$, then the total transform of $\Delta$ is the copy of $\Delta$ in $\Gamma'$ whose vertices are the strict transform of the vertices of $\Delta$. 
If the center is a vertex of $\Delta$, then its total transform is obtained from $\Delta$ by adding the new edge and the new vertex. 
If the center is an edge $\edge{v}{w}$ of $\Delta$, then its total transform is the graph whose vertices are the vertices of $\Delta$ together with the new vertex, 
and edges are edges of $\Delta$ different from $\edge{v}{w}$ together with the two new edges. 
Finally, if the center is an edge $\edge{v}{w}$ with $v\in \Delta$ and $w\notin \Delta$, then the vertices of the total transform are the vertices of $\Delta$ together with the new vertex, and edges are edges of $\Delta$ together with the only new edge that contains $v$.
It is not difficult to see that the total transform of a connected subgraph of $\Gamma$ is a connected subgraph of $\Gamma'$.


\smallskip

We also introduce the notion of \emph{embedding} of a (weighted) graph of $\Gamma$ into another one $\Gamma'$.
This is an injective map $\varphi\colon V(\Gamma) \to V(\Gamma')$ such that $v \sim w$ implies $\varphi(v) \sim \varphi(w)$, $g'(\varphi(v))= g(v)$ and $e'(\varphi(v)) =e(v)$.
We shall write $\varphi \colon \Gamma \hookrightarrow \Gamma'$, and  denote by $\varphi(\Gamma)$ the subgraph of $\Gamma'$ whose vertices are $\varphi(v)$ with $v\in V(\Gamma)$ and edges $\edge{\varphi(v)}{\varphi(w)}$
with $\edge{v}{w} \in E(\Gamma)$. An \emph{isomorphism} is an embedding such that the induced maps on vertices and edges are bijective.

\smallskip

Suppose we are given two modifications $\Gamma' \rightsquigarrow \Gamma$ and  $\Delta' \rightsquigarrow \Delta$ 
and two embeddings $\varphi \colon \Delta \hookrightarrow \Gamma$, $\varphi' \colon \Delta' \hookrightarrow \Gamma'$. Then we say
that  the pair $(\Gamma' \rightsquigarrow \Gamma, \Delta' \rightsquigarrow \Delta)$ is $(\varphi', \varphi)$-\emph{compatible} whenever there exist simple modifications $\Gamma_{i+1} \rightsquigarrow \Gamma_i$
and embeddings $\varphi_i: \Delta_i \hookrightarrow \Gamma_i$ such that 
$\Gamma_0 = \Gamma$, $\Gamma_l = \Gamma'$, $\Delta_0 = \Delta$, $\Delta_l = \Delta'$, 
$\varphi_0 = \varphi$, $\varphi_l = \varphi'$, and the following holds.
\begin{enumerate}[label=(\roman*)]
\item\label{item:compatible_rules_vin}
When $\Gamma_{i+1} \rightsquigarrow \Gamma_i$ is centered at a vertex $\varphi_i(v)$ for some 
$v\in \Delta_i$, then $\Delta_{i+1}$ is obtained from $\Delta_i$ by the simple modification centered at $v$. 
\item
When $\Gamma_{i+1} \rightsquigarrow \Gamma_i$ is centered at a vertex $v\notin  \varphi_i(\Delta_i)$, then $\Delta_{i+1} = \Delta_i$.
\item 
When $\Gamma_{i+1} \rightsquigarrow \Gamma_i$ is centered at an edge $\edge{\varphi_i(v)}{\varphi_i(w)}$ for some 
edge $\edge{v}{w}$ in $\Delta_i$, then $\Delta_{i+1}$ is obtained from $\Delta_i$ by  the simple modification centered at $\edge{v}{w}$. 
\item 
When $\Gamma_{i+1} \rightsquigarrow \Gamma_i$ is centered at an edge $\edge{v}{\varphi_i(w)}$
for some $v \notin \Delta_i$ and $w\in\Delta_i$, then $\Delta_{i+1}$ is obtained from $\Delta_i$ by the simple modification centered at  $w$. 
\item \label{item:compatible_rules_eout}
When $\Gamma_{i+1} \rightsquigarrow \Gamma_i$ is centered at an edge  $\edge{v}{w}$
for some $v,w \notin \Delta_i$, then $\Delta_{i+1} = \Delta_i$.
 \end{enumerate}

\begin{prop}\label{prop-modif}
Suppose the pair $(\Gamma' \rightsquigarrow \Gamma, \Delta' \rightsquigarrow \Delta)$ is $(\varphi', \varphi)$-compatible, where
$\Gamma' \rightsquigarrow \Gamma$ and  $\Delta' \rightsquigarrow \Delta$  are modifications, and $\varphi \colon \Delta \hookrightarrow \Gamma$, $\varphi' \colon \Delta' \hookrightarrow \Gamma'$ are embeddings. 
\begin{enumerate}[label=(\roman*)]
\item\label{item:prop-modif-1}
The image by $\varphi'$ of the total transform in $\Delta'$ of any subgraph $G\subset \Delta$ is the total transform of $\varphi(G)$.
In particular, the total transform of $\varphi(\Delta)$ is $\varphi'(\Delta')$.
\item\label{item:prop-modif-2}
If $\imath_\Delta$ and $\imath_{\Gamma}$ denote the strict transform maps induced by the modifications, then 
$ \varphi' \circ \imath_\Delta = \imath_{\Gamma} \circ \varphi $.
\end{enumerate}
\end{prop}
\begin{proof}
It is only necessary to check these properties when $\Gamma' \rightsquigarrow \Gamma$ is a simple modification in which case $\Delta' \rightsquigarrow \Delta$
is described by one of the rules \ref{item:compatible_rules_vin}-\ref{item:compatible_rules_eout} above. It is a routine argument to verify the proposition in each of these cases. 
\end{proof}

\begin{prop}\label{prop-induction}
Suppose $\Delta' \rightsquigarrow \Delta$ is a modification, and let $\varphi \colon\Delta\hookrightarrow\Gamma$ be any embedding. 
\begin{itemize}
\item
Then there exists a modification  $\Gamma' \rightsquigarrow \Gamma$, called the \emph{induced modification}, and an embedding $\varphi' \colon\Delta'\hookrightarrow\Gamma'$
such that the pair $(\Gamma' \rightsquigarrow \Gamma, \Delta' \rightsquigarrow \Delta)$ is $(\varphi', \varphi)$-compatible.
\item
Suppose we are given another embedding $\psi\colon \Gamma \hookrightarrow \widetilde{\Gamma}$. 
Then the two modifications of $ \widetilde{\Gamma}$ induced by $\psi$ or $\psi \circ \varphi$ are isomorphic. 
We denote either of them by $\widetilde{\Gamma}'\rightsquigarrow\widetilde{\Gamma}$. 
The embedding $\Delta'\hookrightarrow\widetilde{\Gamma}'$ is the composition of the embeddings $\Delta'\hookrightarrow\Gamma'$ and $\Gamma'\hookrightarrow\widetilde{\Gamma}'$. 
\end{itemize}
\end{prop}

\begin{proof}
Decompose $\Delta' \rightsquigarrow \Delta$  into simple modifications $\Delta_{i+1} \to \Delta_i$,  such that 
$\Delta_0 = \Delta$, $\Delta_l = \Delta'$. To simplify notation we assume $\Delta_0 \subset \Gamma_0 := \Gamma$.
We define by induction a graph $\Gamma_i$ that contains $\Delta_i$ as follows. 
If $\Delta_{i+1}$ is the modification centered at an edge $\edge{v}{w}$ with $v,w \in \Delta_i$, then set $\Gamma_i$
to be the modification centered at $\edge{v}{w}$. If $\Delta_{i+1}$ is the modification centered at a vertex $v\in \Delta_i$, then we set $\Gamma_i$
to be the modification centered at $v$. It is clear from the definitions that we obtain compatible modifications in this way.
The second point is proven in the same way; 
details are left to the reader. 
\end{proof}

\begin{rmk}
A graph $\Gamma'$ satisfying the first point of Proposition~\ref{prop-induction} is in general not unique. When $\Delta_{i+1}$ is the modification centered at a vertex $v\in \Delta_i$ that is connected to a vertex
$w\notin \Delta_i$, we could also have defined $\Gamma_{i+1}$ to be the modification centered at the edge $\edge{v}{w}$.
Note however that the construction  given  in the proof of the previous proposition yields a graph $\Gamma'$ that is minimal in the sense that the sum of the weights at all its vertices is minimal (among all possible graphs). 
\end{rmk}

\subsection{Characterization of self-similar graphs}

\begin{defi}\label{def:graph-ss}
A connected graph is said to be \emph{regular} if it is  a modification of the graph with one vertex and weight $(0,1)$. 
It is said to be 
\emph{sandwiched} if it can be embedded into a regular graph. 

A connected graph $\Gamma$ is said to be \emph{self-similar} if it admits a nontrivial modification $\Gamma'$
that contains a subgraph $\Gamma_0$ that is isomorphic (as a weighted graph) to $\Gamma$.
\end{defi}

\begin{thm}\label{thm:graph-sdw}
Suppose $\Gamma$ is a connected weighted graph. 
If $\Gamma$ is self-similar then it is sandwiched. 
\end{thm}

In fact the reverse implication is also true. We leave it as an exercise to the reader since we shall not need it in the sequel.

\begin{proof}
We fix a modification $\Gamma'$ of $\Gamma$, and an embedding $\varphi\colon \Gamma \hookrightarrow \Gamma'$. Recall that the strict transform yields an inclusion 
$\imath\colon V(\Gamma) \to V(\Gamma')$.  

\medskip

\noindent {\bf Step 1}. There exists a sequence of graphs $\Gamma_n$ with $\Gamma_0 := \Gamma$, and $\Gamma_1 := \Gamma'$,  such that 
\begin{enumerate}[label=(\roman*)]
\item $\Gamma_{n+1}$ is a modification of $\Gamma_n$;
\item there exist embeddings $\varphi_n\colon \Gamma_n \hookrightarrow \Gamma_{n+1}$;
\item \label{item:step1compatibility} the pair $(\Gamma_{n+1}\rightsquigarrow \Gamma_n,\Gamma_{n}\rightsquigarrow \Gamma_{n-1})$ is $(\varphi_n, \varphi_{n-1})$-compatible.
\end{enumerate}
Since $\Gamma'\rightsquigarrow \Gamma$ is a modification and $\Gamma$ is a subgraph of $\Gamma'$ via the embedding $\varphi$,
Proposition~\ref{prop-induction} gives a modification $\Gamma_2\rightsquigarrow \Gamma'$ and an embedding $\varphi_1\colon \Gamma' \hookrightarrow \Gamma_2$
with the right compatibility properties.  

To  build the sequence $\Gamma_n$ we proceed in the same way and use repeteadly Proposition~\ref{prop-induction}. 
More precisely suppose $\Gamma_{n}$ has been defined. For each $n$ there exists an embedding $\Gamma \hookrightarrow \Gamma_n$ obtained by the map 
$\Phi_n := \varphi_{n-1} \circ \ldots \circ \varphi_0$ (with $\varphi = \varphi_0$). Then we construct
$\Gamma_{n+1}$ from this embedding and the modification $\Gamma'\rightsquigarrow \Gamma$ using Proposition~\ref{prop-induction}.

Observe that $\Gamma_{n}$ is obtained by the modification $\Gamma'\rightsquigarrow \Gamma$ using $\Phi_{n-1}$ whereas 
$\Gamma_{n+1}$ is obtained by the modification $\Gamma'\rightsquigarrow \Gamma$ using $\varphi_{n-1} \circ \Phi_{n-1}$. 
This implies the existence of an embedding $\varphi_n \colon \Gamma_n \to \Gamma_{n+1}$ satisfying the compatibility property (c)
(see the second point of Proposition~\ref{prop-induction}).

\medskip

\noindent {\bf Step 2}.  Suppose that there exists a vertex $v\in V(\Gamma)$ such that $\imath(v) = \varphi(v)$. We claim that the modification is trivial. 

\smallskip

To see this define the edge distance on $V(\Gamma)$ by setting $d(x, y)$ to be the least integer $n \in \N$ such that there exists 
$v_0, \ldots, v_n \in V(\Gamma)$ with the property  $v_0 =x$, $v_n=y$ and $\edge{v_0}{v_1}, \ldots, \edge{v_{n-1}}{v_n} \in E(\Gamma)$. 

\smallskip

We first prove by induction on $n:= d(w, v)$, that  there exists an integer $k\ge1$ such that  
$I_k(w) = \Phi_k(w)$ for any $w\in V(\Gamma)$ at distance at most $n$ to $v$ where $I_k$ is the strict transform map induced by the modification
$\Gamma_k \rightsquigarrow \Gamma$, and $\Phi_k := \varphi_{k-1} \circ \ldots \circ \varphi_0$.
 
\smallskip

Our claim for $n=0$ is our standing assumption (with $k=1$).  Assume the induction hypothesis for vertices at distance $n-1$  to $v$. 
 Pick  $w$ at distance $n$ to $v$. By definition there exists (at least one) $w_0\in V$ at distance $n-1$ to $v$ and $1$ to $w$. Replacing $\Gamma_k$ by $\Gamma'$
we may suppose that $\imath(w_0) = \varphi(w_0)$. In particular we have  $e(\imath(w_0)) = e(w_0)$. 

Decompose the modification from  $\Gamma' \rightsquigarrow\Gamma'$ into a sequence of simple modifications
$\Delta_{i+1} \to \Delta_i$ with   $\Delta_0:=\Gamma$,  and $\Gamma':= \Delta_l$.
The equality of weights $e(\imath(w_0)) = e(w_0)$ implies that the center of all simple modifications
$\Delta_{i+1} \to \Delta_i$ cannot be an edge that contains the strict transform of $w_0$ in $\Delta_i$ or a vertex that is joined to this strict transform
by an edge.  It follows that the strict transform of any vertex  at distance $1$ from $ w_0$ in $\Gamma'$ is again joined to $\imath(w_0)$ by an edge, and no new edge
starting from $\imath(w_0)$ are created. In particular, $\imath$ induces a bijection from the link $L(w_0)$ to $L(\imath(w_0))$, and 
the valency of $w_0$ is equal to the valency of $\imath(w_0)$. Since $\varphi$ is an isomorphism from $\Gamma$ onto its image,  $\varphi(L(w_0))$ is included in $L(\varphi(w_0))$. The two sets having the same cardinality, we conclude that $\varphi$ is a bijection (hence an isomorphism) from $L(w_0)$ to $L(\varphi(w_0))=\varphi(L(w_0))$.  Denote by $f := \imath^{-1} \circ \varphi$ the bijection on $L(w_0)$.

\smallskip

Observe  that the quantity $\sum_{v\in L(w_0)} e(v)$ is equal to $\sum_{v\in L(\varphi(w_0))} e(v)$. 
Since this quantity can only increase along the sequence of simple modifications $\Delta_{i+1} \to \Delta_i$
it remains constant.  This implies that  the total transform of $L(w_0)$ is equal to $L(\varphi(w_0))$. 

\smallskip

By Proposition~\ref{prop-modif} \ref{item:prop-modif-1} and  Property \ref{item:step1compatibility} of Step 1, it follows that the total transform
of $L(\varphi(w_0)) = \varphi(L(w_0))$ by $\Gamma_2\rightsquigarrow \Gamma_1$ is also equal to $\varphi_1( \varphi(L(w_0)))$. 
In other words, the total transform of $L(w_0)$ by the composite modification $\Gamma_2 \rightsquigarrow \Gamma$
is isomorphic to itself.  In particular the injective map $I_2$ induced by this strict transform defines a bijection from
$L(w_0)$ to $\varphi_1( \varphi(L(w_0)))$, and we get as before a bijection $f_2\colon L(w_0) \to L(w_0)$.

Denote by $\imath_1$ the strict transform map induced by $\Gamma_2 \rightsquigarrow \Gamma$ so that 
$f_2 = \imath^{-1} \circ \imath_1^{-1}\circ  \varphi_1 \circ \varphi$.
By Proposition~\ref{prop-modif} \ref{item:prop-modif-2}, we have $\varphi_1 \circ  \imath= \imath_1\circ  \varphi$, 
whence $f_2 =  \imath^{-1} \circ  \varphi\circ \imath^{-1} \circ  \varphi = f \circ f$.

By repeating this argument, we see that the strict tranform of $L(w_0)$ by the modification $\Gamma_k \rightsquigarrow \Gamma$ is equal to
$\Phi_k(L(w_0))$ and that $I_k^{-1} \circ \Phi_k = f^{\circ k}$. 
Since $L(w_0)$ is finite, it follows that $f^{\circ k} = \id$ for some $k$. 
Observe that we may choose $k$ to be the least common multiple of all integers less than the cardinality of $\Gamma$, which 
is then independent of $w_0$. 
This proves our claim.

\smallskip

We now explain how this claim implies Step 2.  Replacing $\Gamma'$ by some $\Gamma_k$ for $k$ sufficiently large, 
we have $\imath(w) = \varphi(w)$ for all $w$. Since $\varphi$ preserves the weights, we get $e(\imath(w)) = e(w)$ for all $w$
which is only possible when the modification is trivial.

\medskip

\noindent {\bf Step 3}. If $\Gamma$ is self-similar, 
then there exists an integer  $k\ge 1$ such that for all $v\in V(\Gamma)$  one has $I_k(v) \notin \Phi_k(V(\Gamma))$.

\smallskip

We first prove that  if $I_k(v) \notin \Phi_k(V(\Gamma))$ then $I_{k+l}(v) \notin \Phi_{k+l}(V(\Gamma))$ for all $l\ge 0$. 
Indeed suppose $I_k(v) \notin \Phi_k(V(\Gamma))$, and recall from Step 1 that $\Gamma_{k+1}$ is built from the embedding $\Phi_k \colon \Gamma \hookrightarrow \Gamma_k$
and the modification $\Gamma' \rightsquigarrow \Gamma$. It follows that the strict tranform $I_{k+1}(v)$ of $I_k(v)$ in $\Gamma_{k+1}$ is not contained
in the image of $\Gamma'$ in $\Gamma_{k+1}$ which contains the image of $\Gamma$ in $\Gamma_{k+1}$. 
Therefore $I_{k+1}(v) \notin \Phi_{k+1}(V(\Gamma))$ which proves $I_{k+l}(v) \notin \Phi_{k+l}(V(\Gamma))$ for all $l\ge 0$ by induction as required. 

\smallskip

It is thus sufficient to find for any vertex $v$ of $\Gamma$ an integer $k$ such that $I_k(v) \notin \Phi_k(V(\Gamma))$.
To prove this we proceed by contradiction,  and pick a vertex $v$ of $\Gamma$ for which $I_k(v) \in \Phi_k(V(\Gamma))$ for all $k$. 
Since $\Phi_k$ is an embedding, for each $k$ there exists a unique vertex $v_k$ of $\Gamma$ such that $I_k(v) = \Phi_k(v_k)$.
Choose $k, l >0$ such that $v_{k+l} = v_k$. 

Observe that we have a modification $\Gamma_{k+l}\rightsquigarrow \Gamma_k$ and an embedding $\Phi \colon\Gamma_k \hookrightarrow\Gamma_{k+l} $
obtained by composing $\varphi_{k+l-1} \circ \ldots \circ \varphi_k$. Denote by $I$ the strict transform map induced by $\Gamma_{k+l}\rightsquigarrow \Gamma_k$, and let $w := \Phi_k(v_k) = \Phi_k(v_{k+l}) = I_k(v)$. 
Then we get
\[
I(w) = \imath_{k+l-1} \circ \ldots \circ \imath_k (I_k(v)) = I_{k+l}(v) = \Phi_{k+l}(v_{k+l}) = \Phi (\Phi_k(v_{k+l})) = \Phi(w)~.
\] 
Step 2 then implies that $\Gamma_{k+l}\rightsquigarrow \Gamma_k$ is an isomorphism which is only possible if $\Gamma'\rightsquigarrow \Gamma$
is also an isomorphism. This is a contradiction.

\medskip

\noindent {\bf Step 4}. Finally we prove that $\Gamma$  being self-similar implies the graph to be sandwiched.

\smallskip

By Step 3, replacing $\Gamma'$ by $\Gamma_k$ for a sufficiently large $k$ we may  assume that $\imath(v) \notin \varphi(V(\Gamma))$ for all vertices $v$ of $\Gamma$.
It remains to prove that this implies the graph to be sandwiched.

This is a consequence of the following general fact. Suppose $\Gamma'  \rightsquigarrow \Gamma$ is a modification. 
Look at the (not necessarily connected) graph $\widehat{\Gamma'}$ obtained by removing the strict transforms of all vertices in $\Gamma$
and all edges connected to any  of these. 

\begin{lem}\label{lem:whaou}
The graph $\widehat{\Gamma'}$ is a union of regular graphs.
\end{lem}

Now $\varphi(\Gamma)$ is a subgraph that does not contain any of the strict transform $\imath(v)$ for $v\in \Gamma$, hence
is a subgraph of $\widehat{\Gamma'}$. In particular, $\varphi(\Gamma)$ (hence $\Gamma$) is sandwiched.
\end{proof}

\begin{proof}[Proof of Lemma~\ref{lem:whaou}]
We argue by induction on the number $k$ of simple modifications needed to define $\Gamma'  \rightsquigarrow \Gamma$.
When $k=1$ then $\widehat{\Gamma'}$ is reduced to the sole vertex added to $\Gamma$ and by definition its weight is $(0,1)$, hence is regular.
Suppose the lemma is proved for $k-1$, and suppose $\Gamma'  \rightsquigarrow \Gamma$ is decomposed into
$k$ simple modifications $\Delta_i \rightsquigarrow \Delta_{i-1}$.

Consider the modification $\Delta_{k-1} \rightsquigarrow \Gamma$. By the inductive assumption, the corresponding graph
$\widehat{\Delta}_{k-1}$ is a disjoint union of regular graphs.  When $\Delta_k$ is a modification centered at a vertex that does not lie
in $\widehat{\Delta}_{k-1}$ then $\widehat{\Gamma'}$ is the disjoint union of $\widehat{\Delta}_{k-1}$ and a single vertex with weight $(0,1)$. 
When $\Delta_k$ is  a modification centered at a vertex $v\in \widehat{\Delta}_{k-1}$ then $\widehat{\Gamma'}$ is obtained by a modification of $\widehat{\Delta}_{k-1}$ centered at $v$
and it still regular.  When $\Delta_k$ is a modification centered at an edge $\edge{v}{w}$ with $v,w\notin  \widehat{\Delta}_{k-1}$, then  
$\widehat{\Gamma'}$ is the disjoint union of $\widehat{\Delta}_{k-1}$ and a single vertex with weight $(0,1)$. 
When $\Delta_k$ is a modification centered at an edge $\edge{v}{w}$ with $v\in \widehat{\Delta}_{k-1}$ and $w\notin  \widehat{\Delta}_{k-1}$, then  
$\widehat{\Gamma'}$ is obtained by a modification of $\widehat{\Delta}_{k-1}$ centered at $v$. 
Finally, when $\Delta_k$ is a modification centered at an edge $\edge{v}{w}$ with $v,w\in  \widehat{\Delta}_{k-1}$, then  
$\widehat{\Gamma'}$  is obtained by a modification of $\widehat{\Delta}_{k-1}$ centered at $\edge{v}{w}$.
In all cases, all connected components of $\widehat{\Gamma'}$ remain regular which completes the proof.
\end{proof}



\section{Self-similar links have self-similar skeletons}\label{section_proof_main}

 In this section we prove the implication \ref{condition_selfsim} $\implies$ \ref{condition_graph} of Theorem~\ref{mainthm}.
 For technical reasons we introduce the following condition.

 \begin{enumerate}[label=($\dagger$)]
 	\item \label{condition_prime}   There exists a finite set $T$ of type 1 points of $\NL(X,0)$ such that $\NL(X,0)\setminus T$ is isomorphic to an open subset $U$ of $\NL(X,0)$ satisfying $\overline{U} \subsetneq \NL(X,0)$.
 \end{enumerate}
 \label{conditiondag}

 Observe that \ref{condition_selfsim} implies \ref{condition_prime}, so that 
we are reduced to prove \ref{condition_prime} $\implies$ \ref{condition_graph}. To do so, 
we rely  on the following result.

\subsection{Extending morphisms from the punctured disc}
\begin{prop}\label{lem:keyboundary}
Let $D$ be an open $k((t))$-analytic disc and let $0$ denote its origin.
Let $\varphi \colon D\setminus\{0\}\to\NL(X,0)$ be any map that induces an isomorphism of locally ringed spaces between $D\setminus\{0\}$ and its image. 
Then the map $\varphi$ extends continuously to $0$, and $\varphi(0)$ is either a divisorial point or a rigid point of $\NL(X,0)$. 
In the latter case $\varphi$ extends to an isomorphism of locally ringed spaces from $D$ to $\varphi(D)$.
\end{prop}

Let us first take for granted  the following lemma.
\begin{lem}\label{lem:extension}
	The map $\varphi \colon D\setminus\{0\}\to\NL(X,0)$ extends uniquely to a continuous map $D\to\NL(X,0)$.
\end{lem}

\begin{proof}[Proof of Proposition~\ref{lem:keyboundary}]

We continue to denote by $\varphi$ the continuous map on $D$ obtained via Lemma~\ref{lem:extension}.
We will prove by contradiction that $\varphi(0)$ is either a divisorial point or a rigid point of $\NL(X,0)$.

Suppose that $\varphi(0)$ is a point of type 3 of $\NL(X,0)$.
Then there exist an open subspace $U$ of $\varphi(D\setminus\{0\})$ with $\varphi(0)\in\overline U$ and two non constant functions $f$ and $g$ on $U$ such that
\[
\lim_{y\in U,y\to \varphi(0)}\frac{\log|f(y)|}{\log|g(y)|}
\]
 exists and is an irrational number.
Therefore the same holds for the ratio of the logarithms of the absolute values of the two nonconstant functions $\varphi^*f$ and $\varphi^*g$ on $\varphi^{-1}(U)$, tending towards $0$.
This yields a contradiction since $0$ is a rigid point of $\NL(X,0)$, therefore its valuation ring has rational rank $1$ over $k$.

Now suppose that the point $\varphi(0)$ is of type 4.
Then we can find $U$ as above, bounded functions $\{f_n\}_{n\in\N}$, and $t$ on $U$ such that the group 
\[
\Gamma = 
\bigg\langle 
\lim_{y\in U,y\to \varphi(0)}\frac{\log|f_n(y)|}{\log|t(y)|}
\bigg\rangle_{n\in\N}
\]
is not a finitely generated subgroup of $\Q$. 
Now consider the bounded functions $F_n=\varphi^*f_n$ and $T=\varphi^*t$ on $\varphi^{-1}(U)$.
Since $\varphi(U)\cup\{0\}$ is an open neighborhood of $0$ in $D$ we can find a smaller disc $D'$ centered in $0$ such that all the $F_n$ and $T$ are defined on $D'\setminus\{0\}$.
The functions $F_n$ and $T$ extend to bounded functions on $D'$, therefore they can be expressed as power series in a variable $X$.
Observe now that we have 
\[
\lim_{y\to \varphi(0),y\in U}\frac{\log|f_n(y)|}{\log|t(y)|}=
\lim_{y\to 0,y\in D}\frac{\log|F_n(y)|}{\log|T(y)|}= \frac{\mathrm{val}_X (F_n(X))}{\mathrm{val}_X (T(X))}.
\]
This implies that $\Gamma$ is a subgroup of $\frac{1}{\mathrm{val}_X (T(X)} \Z$, hence it is finitely generated, which is a contradiction.
It follows that $\varphi(0)$ can only be a rigid point or a divisorial point of $\NL(X,0)$.

If $\varphi(0)$ is a rigid point then it has a neighborhood $U$ isomorphic to a disc thanks to Corollary~\ref{corollary_topology}, and without loss of generality we can take a smaller disc $D$ if needed and assume that $\varphi(D)$ is contained in $U$.
We obtain an injective analytic map from a punctured disc to disc, 
thus this map necessarily extends as an isomorphism from $D$ onto its image.
\end{proof}

\begin{proof}[Proof of Lemma~\ref{lem:extension}]
By Corollary~\ref{corollary_global_topology}, the compact set $\NL(X,0)$ is a disjoint union of discs, finitely many annuli and finitely many (type 2) points.
Recall the definition of $x_r$ for $r\in (0,1)$ from Theorem~\ref{thm:R-tree} and define the injective continuous map $\gamma(r) = \varphi(x_r)$ from $(0,1)$ to $\NL(X,0)$. 
Since $\gamma$ is injective, there exists a subset $A$ of $\NL(X,0)$ that is either a disc or an annulus and such that $\gamma(r) \in A$ for all $r$ sufficiently small. 
Recall that  $A$ is uniquely arcwise connected by Theorem~\ref{thm:R-tree} (i), and has either one or two endpoints by Theorem~\ref{thm:R-tree} (ii). 
Moreover, the complement of $A$ in its closure $\overline{A}$ in $\NL(X,0)$ contains either one or two points by Corollary~\ref{corollary_global_topology}.
This implies that $\overline{A}$ is both uniquely arcwise connected and compact.
It follows that $\gamma$ extends continuously as a  map from $[0,1)$ to $\overline{A}$ by setting $\gamma(0) = \lim_{r\to 0} \gamma(x_r)$.
We set $\varphi(0) = \gamma(0)$, and claim that the resulting map $\varphi\colon D \to \overline{A} \subset \NL(X,0)$ is continuous. 
Since $\overline{A}$ and $D$ are uniquely arcwise connected, one may define the segment joining any two points $x, y$, and we denote it by $[x,y]$. 
For any $0< r <  1$, consider the set $U_r \subset \overline{A}$ (resp. $V_r \subset D$) consisting of those points $x$ such that
$\varphi(x_r) \notin [x, \varphi(0)]$ (resp. $x_r \notin [x, 0]$). These sets form bases of neighborhoods of $\varphi(0)$ in $\overline{A}$ and $0$ in $D$ respectively.
Since $\varphi$ is continuous and injective on $D\setminus \{0 \}$, one has $\varphi(V_r) \subset U_r$, which implies the continuity of $\varphi$ at $0$. 
\end{proof}

\subsection{The implication \ref{condition_prime} $\implies$ \ref{condition_graph}}

Let $T$ be a finite and nonempty set of rigid points of $\NL(X,0)$, let $U$ be an open subset of $\NL(X,0)$ whose closure $\overline U$ in $\NL(X,0)$ is strictly contained in $\NL(X,0)$, and let $\varphi\colon \NL(X,0) \setminus T \to U$ be an isomorphism.
We will show that $(X,0)$ has a self-similar dual graph.

Since every point $x$ of $T$ has a neighborhood in $\NL(X,0)$ that is a disc with origin $x$, by repeatedly applying Lemma~\ref{lem:keyboundary} we deduce that $\partial U=\overline U\setminus U$ is a finite subset of $\NL(X,0)$ consisting of rigid and divisorial points.
Moreover, we can extend $\varphi$ to each point of $T$ whose image is a rigid point and therefore assume without loss of generality that $\partial U$ only consists of divisorial points.
Observe also that $\partial U$ is 
nonempty since $\overline U$ is strictly contained in $\NL(X,0)$.

Let $C$ be a germ of curve on $(X,0)$ whose components correspond to the points of $T$, let $\pi \colon X' \to X$ be the good resolution of $(X,0)$ that also resolves the germ $C$ and that is minimal with respect to this property (see e.g. \cite[Theorem 5.12]{laufer:normal2dimsing}), and let $Z$ be another good resolution 
 of $(X,0)$ resolving the germ $C$ and whose divisorial set $\Div(Z)$ contains $\partial U$.

Set $S=\big(\Div(Z) \setminus U \big) \cup \varphi (\Div(X'))$ and let $\Y'$ be the formal modification whose divisorial set is $S$, as given by Theorem~\ref{theorem_existence_modifications}.
We claim that $\Y'$ is algebraized by a good resolution of $(X,0)$ which resolves the germ $C$. 
This boils down to verifying the conditions of Theorem~\ref{theorem_characterization_resolutions}.

In order to do so, let $V$ be any connected component of $\NL(X,0) \setminus S$. 
Recall that $S$ contains $\Div(Z)$ hence $\partial{U}$,
so that $V$ is either disjoint from ${U}$, or contained in it.
In the first case, $V$ is a connected component of $\NL(X,0) \setminus \Div(Z)$. 
Since $Z$ is a good resolution of $(X,0)$ and $C$,  it follows that $V$ has the form we want by Theorem~\ref{theorem_characterization_resolutions}.
In the second case, $V$ is isomorphic through $\varphi$ to a connected component of $\NL(X,0) \setminus (\Div(X') \cup T)$.
As $X'$ is a resolution of $C$, such a component is a disc, an annulus, or a disc with the origin removed (which is itself isomorphic to an annulus, as observed in Remark~\ref{remark_disc_minus_point}).
This completes the proof of the fact that $\Y'$ is algebraized by a good resolution $Y'$ of $(X,0)$ and $C$.

\smallskip

Now let $\Y$ be the formal modification of $(X,0)$ whose divisorial set is $\Div(Z) \setminus U$ and denote by $y$ the point of the reduction $\Y_0$ of $\Y$ corresponding as in Proposition~\ref{proposition_propertiesNL_3} to the connected component $U$ of $\NL(X,0) \setminus \Div(\Y)$.
Let $\cJ_y$ be the ideal of $\widehat{\cO_{\Y,y}}$ which defines $\Y_0$ locally around $y$.
Observe now that we have a sequence of isomorphisms of complete local rings
\begin{multline*}
\widehat{\cO_{\Y,y}} 
\mathop{\cong}\limits^{\mathrm{Prop.}\ref{proposition_propertiesNL_1}} 
\mathcal O_{\NL(\Y,y)}^\circ \big( \NL(\Y,y) \big)
\mathop{\cong}\limits^{\mathrm{Lem.}\ref{lemma_extension_to_type1}}
 \mathcal O_{\NL(\Y,y)}^\circ \big( \NL(\Y,y) \setminus V(\cJ_y)\big)\\
\mathop{\cong}\limits^{\mathrm{Prop.}\ref{proposition_propertiesNL}}
 \mathcal O_{\NL(X,0)}^\circ \big( U \big) 
 \mathop{\cong}\limits^{\varphi}
  \mathcal O_{\NL(X,0)}^\circ \big(\NL(X,0)\setminus T) 
\mathop{\cong}\limits^{\mathrm{Lem.}\ref{lemma_extension_to_type1}}
\mathcal O_{\NL(X,0)}^\circ \big(\NL(X,0) ) 
\mathop{\cong}\limits^{\mathrm{Prop.}\ref{proposition_propertiesNL_1}} 
 \widehat{\cO_{X,0}},
 \end{multline*}
which tells us that $\varphi$ induces an isomorphism of formal schemes $\widehat{\Y/y}\cong \Spf\big(\widehat{\cO_{X,0}}\big)$.

\smallskip

The inclusion $\Div(\Y) \subset \Div(\Y')$ yields a morphism of formal schemes $\psi\colon \Y' \to \Y$ (geometrically this morphism is the contraction of the exceptional components of $\Y'$ corresponding to divisorial points in $U$).
Then $\varphi$ induces an isomorphism of formal schemes between the formal completion $\Y''=\widehat{\Y'/\psi^{-1}(y)}$ of $\Y'$ along $\psi^{-1}(y)$ and the formal completion $\X'=\widehat{X'/\pi^{-1}(0)}$ of $X'$ along its exceptional divisor $\pi^{-1}(0)$.

Since $X'$ is the minimal good resolution of $(X,0)$ and $C$, the resolution $\mu\colon Y'\to X$ factors through a morphism $\rho \colon Y' \to X'$.
Now observe that since $\widehat{\Y'/\psi^{-1}(y)}$ and $\widehat{X'/\pi^{-1}(0)}$ are isomorphic as formal schemes, then $\Dual\big(\psi^{-1}(y)\big)$ and $\Dual\big(\pi^{-1}(0)\big)$ are isomorphic as weighted graphs (see section \ref{ssec:dualgraphs} for the definition of the weighted graph associated to a good resolution).
Indeed, let $x$ be a point of $\Div(X')$ corresponding to a component $E$ of $\pi^{-1}(0)$ and denote by $\varphi(E)$ the component of $\psi^{-1}(y)$ corresponding to the point $\varphi(x)\in \Div(Y')$.
Then we have the following sequence of field isomorphisms 
\[
k(E)\cong\widetilde{\mathscr H(x)}\cong \widetilde{\mathscr H\big(\varphi(x)\big)} \cong k\big(\varphi(E)\big)
\]
since $\varphi$ is an isomorphism of locally ringed spaces, hence $E$ and $\varphi(E)$ have the same genus.
Moreover, we have $\widehat{Y'/\varphi(E)}=\widehat{\Y''/\varphi(E)} \cong \widehat{\X'/E}=\widehat{X'/E}$, which implies that $(E\cdot E)=(\varphi(E)\cdot\varphi(E))$ because the degree of the normal bundle of $E$ in $X'$ is the same as the degree of the formal normal bundle of $E$ in $\X'$, 
and the same holds for the normal bundle of $\varphi(E)$ in $Y'$.

We now observe that the weighted dual graph $\Dual\big(\psi^{-1}(y)\big)$ is a subgraph of $\Dual\big((\pi\circ\rho)^{-1}(0)\big)$, and the latter is obtained by a sequence of simple modifications from $\Dual\big(\pi^{-1}(0)\big)$ since $\rho \colon Y' \to X'$ is a birational morphism between smooth surfaces. 
This shows that $(X,0)$ has a self-similar dual graph, which is what we wanted to prove.
\hfill \qed

\bigskip

For the reader convenience, the following diagram summarizes the constructions made in the course of the proof above.

\begin{displaymath}
\xymatrix@C=1pc@R=1.3pc@M=3pt@L=3pt{
  & Y' \ar@{->}[d]^{\rho} \ar@{->}[dl]_{\psi} && \widehat{{Y'}/{(\pi\circ\rho)^{-1}(0)}}  \ar@{->}[ll] \ar@{->}[d]^{\cong} && \widehat{Y'/\psi^{-1}(y)}  \ar@{_{(}->}[ll] \ar@{->}[dll]^{\cong}_{\varphi}\\
Y \ar@{->}[dr] & X' \ar@{->}[d]^{\pi} && \widehat{X'/\pi^{-1}(0)} \ar@{->}[ll]\\
  & X
}
\end{displaymath}


\section{Sandwiched singularities are determined by their dual graphs}\label{section_plumbing}
The goal of this section is two-fold. 
We first prove that a normal surface singularity is sandwiched if 
 its (weighted) dual graph is sandwiched.
This is an extension of \cite[Proposition~1.11]{spivakovsky:sandsingdesingsurfNashtransf} to arbitrary characteristic. 
We then explain how this allows to prove the implication \ref{condition_graph} $\implies$ \ref{condition_sandwiched} of Theorem~\ref{mainthm}.

\subsection{Dual graphs of sandwiched singularities}\label{ssec:dualgraphs}

Let $(X,0)$ be any normal surface singularity, and let $(Y,D) \to (X,0)$ be a good resolution of the singularities of $(X,0)$. 
We define the weighted dual graph $\Dual (D)$ as follows. 
Its set of vertices is the set of the irreducible components of $D$, and there is an edge connecting two vertices if and only if the corresponding components intersect. 
The weight of a vertex is the pair of positive integers $\big(g(E),-E^2\big)$ consisting of the genus and of the opposite of the self-intersection of the corresponding component $E$ of $D$.

\begin{thm}\label{thm:extend-spiv}
Let $(X,0)$ be a normal surface singularity and assume there exists a good resolution of $(X,0)$ whose associated weighted dual graph is sandwiched.
Then $(X,0)$ is sandwiched.
\end{thm}

Over the complex numbers, this result is due to Spivakovsky.
His proof is complex analytic in nature, relying on plumbing constructions in an essential way. 
We proceed in very much the same way, using an analogue of plumbing in formal geometry.

\begin{proof}
Suppose that 
$(Y,D)$ is a good resolution of singularities of $(X,0)$, that the associated weighted dual graph $\Dual(D)$ is sandwiched, and choose an embedding of $\Dual(D)$ in a regular weighted graph $\Gamma$.
We can assume without loss of generality that each of the vertices of $\Gamma\setminus\Dual(D)$ has weight $(0,1)$ and valence 1, so that it is only adjacent to a vertex of $\Dual(D)$.
Indeed, each connected component of $\Gamma\setminus\Dual(D)$ is itself regular and can be contracted to a single vertex adjacent to $\Dual(D)$ and of weight $(0,1)$, see \cite[Proposition 1.13]{spivakovsky:sandsingdesingsurfNashtransf}).

We will now build a smooth formal $k$-scheme $\Y'$ whose reduction $\Y'_0$ is a curve with simple normal crossings such that $\Dual(\Y'_0)\cong\Gamma$, together with a closed immersion of formal schemes $\widehat{Y/D} \to \Y'$. 
We will do so by means of a patching procedure, using as elementary building blocks the following smooth formal $k$-schemes of dimension  2.
For $n>0$, let $\mathbb{F}_n = \mathbb{P}_{\mathbb P^1_k} \big( \cO_{\mathbb P^1_k} \oplus \cO_{\mathbb P^1_k}(n) \big)$ be the Hirzebruch surface of order $n$ over $k$ and let $F$ denote the rational curve of self-intersection $-n$ in $\mathbb F_n$.
Pick a point $p$ in $\mathbb F_n$ that does not lie in $F$, call $F'$ the curve in the canonical rational fibration passing through $p$, and denote by $\widetilde{\mathbb F_n}$ the blowup of $\mathbb F_n$ at $p$.
Define $\X_{n}$ to be the formal completion of $\mathbb F_n'$ along the union of $F$ with the strict transform of $F'$ in $\widetilde{\mathbb F_n}$, so that the reduction of $\X_n$ consists of two smooth and rational curves that intersect transversally and have self-intersection $-n$ and $-1$ respectively.
Now take a vertex $v$ in $\Gamma\setminus \Dual(D)$ and let $E$ be the component of $D$ corresponding to the vertex of $\Dual(D)$ adjacent to $v$.
Since $\Gamma$ is regular, $E$ is rational and has negative self-intersection, say $E^2=-m<0$.
Consider the formal $k$-scheme $\X_m$ constructed above and denote by $F$ its rational curve of self-intersection $-m$ and by $q$ the point of $F$ that intersects the other component of $(\X_m)_0$.
The formal completions $\widehat{Y/E}$ and $\widehat{\X_m/F}$ are isomorphic by \cite[Theorem 2.11]{LeeNakayama2013} (in characteristic zero this result was obtained earlier in \cite[Satz 2.10]{Brieskorn1967}), and, after composing with an automorphism of  $\widehat{\X_m/F}$ if necessary (for instance one induced by a suitable automorphism of the Hirzeburch surface), we can choose such an isomorphism $\widehat{\X_m/F}\cong\widehat{Y/E}$ sending $q$ to a point of $E$ that is smooth in $D$.
Therefore we can glue a copy of $\X_m$ to $\widehat{Y/D}$, obtaining a smooth formal $k$-scheme whose associated dual graph is the weighted graph spanned by $\Dual(D)$ and the vertex $v$.
By repeating this procedure for every vertex in $\Gamma\setminus \Dual(D)$ we obtain the formal scheme $\Y'$ as we wanted.
Observe that by construction the formal scheme $\widehat{Y/D}$ is isomorphic to the closed formal subscheme of $\Y'$ obtained by formally completing the latter along $D$.

Since $\Gamma\cong\Dual(\Y_0)$ is regular, Grauert-Artin contractibility criterion \cite{artin:contractibilityalgebraicspaces} gives a formal modification $\pi\colon \Y'\to \mathcal{Z}$ contracting $\Y_0$ to a point $\{z\}=\mathcal Z_0$ in a smooth two-dimensional formal $k$-scheme $\mathcal Z$. 
Having only one point, $\mathcal Z$ is affine, therefore it can be algebraized by a smooth surface $Z$ over $k$.
It follows that $\Y'$, being a formal good resolution of $(Z,z)$, is itself algebraized by a good resolution $\pi\colon Y'\to Z$ by \cite[Proposition 7.6]{fantini:normspaces}.

Using the other contractibility criterion by Artin, \cite[Theorem 2.3]{artin:contractibilitycriterion}, one can then also contract the divisor $D$ of $Y'$ to a (possibly singular) point $z'$ of a normal surface $Z'$ over $k$, which yields a modification $\varpi\colon Y'\to Z'$ such that $\pi$ factors as $Y' \stackrel{\varpi}{\longrightarrow} Z' \stackrel{\tau}{\longrightarrow} {Z}$. 
Observe that the local ring $\widehat{\cO_{Z',z'}}$ is sandwiched since $Z$ is smooth. 

Finally, observe that we have
\begin{displaymath}
\widehat{\cO_{X,0}}
	\cong 
\cO^\circ\big(\NL(X,0)\big)
	\cong
\cO^\circ\big(\NL(Y,D)\big)
	\cong
\cO^\circ\big(\NL(Y',D)\big)
	\cong
\cO^\circ\big(\NL(Z',z')\big)
	\cong
\widehat{\cO_{Z',z'}},
\end{displaymath}
where the first and the last isomorphisms come from Proposition~\ref{proposition_propertiesNL_1} and the others follow from the invariance of non-archimedean links under modifications.
This proves that $\widehat{\cO_{X,0}}$ is sandwiched, which is what we wanted to show.
\end{proof}

\begin{rmk}\label{rem-equivalence}
	The converse to the theorem is also true: the dual graph $\Dual(D)$ is sandwiched if $(X,0)$ is sandwiched. 
	A direct proof of this fact will be given in Remark~\ref{rem:64}.
\end{rmk}

\subsection{The implication \ref{condition_graph} $\implies$ \ref{condition_sandwiched}}

Let $(X,0)$ be any normal surface singularity such that $\Dual(D)$ is self-similar for some
good resolution of singularities $(Y,D) \to (X,0)$.  
By Theorem~\ref{thm:graph-sdw}, the weighted graph $\Dual(D)$ is sandwiched,
which implies that $(X,0)$ is sandwiched by Theorem~\ref{thm:extend-spiv}. \hfill$\qed$



\section{End of the proof of Theorem~\ref{mainthm}: Kato data}\label{sec-Kato data}
In this section, we prove the two implications  \ref{condition_sandwiched} $\implies$  \ref{condition_kato}, and 
 \ref{condition_kato}  implies the condition \ref{condition_prime} from \S\ref{section_proof_main}, 
concluding the proof of our main theorem. 

\subsection{The implications \ref{condition_sandwiched} $\implies$  \ref{condition_kato}}\label{sec:164}
A \textit{Kato datum} \label{def:kato datum} for a normal surface singularity $(X,0)$ is a modification $\pi \colon (X',D) \to (X,0)$, together with an  isomorphism of complete local rings $\widehat{\mathcal O_{X',p}} \cong \widehat{\mathcal O_{X,0}}$ for some $p\in D$.
Note that in particular $\pi$ is not an isomorphism.

Suppose $(X,0)$ is a sandwiched singularity. 
Proving \ref{condition_sandwiched} $\implies$  \ref{condition_kato}  amounts to constructing a Kato datum for $(X,0)$. 
To do so, recall that we may find a proper birational morphism $h \colon Z \to \A^2_k$ and a point $z \in h^{-1}(0)$ such that 
$\widehat{\mathcal O_{Z,z}} \cong \widehat{\mathcal O_{X,0}}$. 
Choose a good resolution of singularities $g \colon Y \to X$ of $(X,0)$, and pick an arbitrary point $q\in g^{-1}(0)$.
Since $Y$ is smooth at $q$ we can find an \'etale map $\varphi \colon Y \to \A^2_k$ mapping $q$ to the origin. 
We  consider the fibered product $X' = Y \times_{\A^2_k} Z$, so that the following diagram is commutative
\[
\xymatrix{
X' = Y \times_{\A^2_k} Z
\ar[r]^-\varpi \ar[d]^\psi & Y \ar[d]^\varphi \ar[r]^g & X\\
Z \ar[r]^h & \A^2_k &
}
\]
Observe that the projection map $\varpi\colon X' \to Y$ is birational since $h$ is, and $\psi \colon X' \to Z$ is \'etale  since \'etale morphisms are preserved by base change. 
The image of $X'$ in $Z$ contains $h^{-1}(0)$ since $\varphi(q) =0$. 
We may thus find a point $p\in X'$ over $q$ such that $\psi(p) = z$. 
Since $\psi$ is \'etale it induces an isomorphism of complete local rings  $\widehat{\mathcal O_{X',p}} \cong \widehat{\mathcal O_{Z,z}}$, and the latter is isomorphic to $\widehat{\mathcal O_{X,0}}$.
It follows that the composition $g \circ \varpi \colon X' \to X$, which is a modification (that is, $(g \circ \varpi)^{-1}(0)$ is a Cartier divisor) because it factors through the resolution $Y$, defines a Kato datum for $(X,0)$.  \hfill$\qed$

\subsection{The implication \ref{condition_kato} $\implies$ \ref{condition_prime}}\label{sec:165} 
Suppose that $(X,0)$ is a normal surface singularity and that $\pi \colon (X',D) \to (X,0)$ is a Kato datum for $(X,0)$, with $p\in D$ a point such that $\widehat{\mathcal O_{X',p}} \cong \widehat{\mathcal O_{X,0}}$.
By Proposition~\ref{proposition_propertiesNL_3} there exists a finite set $T'$ of type 1 points of $\NL(X',p)$ such that the open subspace $U=\mathrm{c}_{X'}^{-1}(p)$ of $\NL(X,0)$ is isomorphic to $\NL(X',p)\setminus T'$.
Since $\widehat{\mathcal O_{X',p}} \cong \widehat{\mathcal O_{X,0}}$, it follows that $U\cong \NL(X,0) \setminus T$ for some finite set $T$ of type 1 points of $\NL(X,0)$.
Finally, the closure $\overline{U}$ of $U$ in $\NL(X,0)$ is strictly contained in the latter, since it is contained in $U\cup \Div(X')$ as observed in Remark~\ref{remark_closure_component}.
\hfill$\qed$

\begin{rmk}\label{rem:64}
Let us sketch a proof of the implication \ref{condition_kato} $\implies$ \ref{condition_graph}. 
Although not necessary to complete the proof of Theorem~A, we obtain in this way a direct argument showing the implication \ref{condition_sandwiched} $\implies$ \ref{condition_graph}, which is equivalent to saying that every sandwiched singularity has a sandwiched dual graph, as mentioned in Remark~\ref{rem-equivalence}.
Suppose that $(X,0)$ is a normal surface singularity and that $\pi \colon (X',D) \to (X,0)$ is a Kato datum for $(X,0)$, with $p\in D$ a point such that $\widehat{\mathcal O_{X',p}} \cong \widehat{\mathcal O_{X,0}}$.
Without loss of generality we may assume that $D$ has simple normal crossings away from $p$, so that in particular $X'$ is smooth away from $p$.
Let $\mu \colon Y \to X$ and $\mu' \colon Y' \to X'$ be the minimal good resolutions of the singularities of $(X,0)$ and $(X',p)$ respectively.  
Since $Y'$ is a good resolution of $(X',p')$ it is also a good resolution of $(X,0)$ so that there exists a birational morphism $\pi' \colon Y' \to Y$ satisfying $\mu \circ \pi' = \pi \circ \mu'$.
This morphism is not an isomorphism, or $\pi$ would be an isomorphism as well.
Moreover, since $Y'$ and $Y$ are smooth, $\pi'$ is a composition of point blow-ups. 
The graph $\Dual\big((\pi\circ\mu')^{-1}(0)\big)$ is thus a nontrivial modification of $\Dual\big(\mu^{-1}(0)\big)$ and contains $\Dual\big((\mu')^{-1}(p)\big)$ which is isomorphic (as a weighted graph) to $\Dual\big(\mu^{-1}(0)\big)$. 
This proves that the latter graph is sandwiched as required.
\end{rmk}



\section{Complex analytic sandwiched singularities}
\label{section_complexanalytic}
In this section we  work with \emph{complex analytic varieties}.
In this context, a sandwiched singularity $(X,0)$ is a germ of normal complex analytic surface such that the completion of its analytic local ring is sandwiched in the sense of Definition~\ref{def-sandwich}.
Our aim is to characterize such singularities in terms of their complex link, proving Theorems~\ref{thm3} and~\ref{thm2}.

\subsection{Pseudoconvex $3$-folds} \label{sec:pdcvx3folds}

Let us recall some definitions and terminology from~\cite{kato:cpctcplxsurfwithGSPH}.
A \emph{real-analytic strongly pseudoconvex $3$-fold} $\Sigma$ is a smooth (real-analytic) hypersurface in a smooth complex surface $S$ such that for any $p\in \Sigma$ there exists an open subset $U$of $S$ containing $p$ and a real-analytic strictly plurisubharmonic function $\rho \colon U \to \R$ such that $\Sigma \cap U = \rho^{-1}(0)$.
A real-analytic strongly pseudoconvex $3$-fold \emph{bounds a Stein domain} if $\Sigma$ is compact and there exists an embedding of a tubular neighborhood of $\Sigma$ in $S$ into
a normal (but possibly singular) complex surface $X$ such that $X \setminus \Sigma$ has one component which is Stein. 

We shall need one more piece of terminology.
A compact real $3$-fold $\Sigma$ in a compact complex surface $S$ is said to be \emph{global} when $S\setminus \Sigma$ is connected. 

\smallskip

Archimedean links are the main example of real-analytic strongly pseudoconvex $3$-folds bounding a Stein domain. 
Let $(X,0)$ be a normal surface singularity, and fix an embedding of $(X,0)$ inside the unit ball in some complex affine space $\C^n$ such that $0$ is sent to the origin.
The function $\rho(z) = \rho(z_1, \ldots , z_n) := \sum_{i=1}^n |z_i|^2$ is real-analytic and strictly plurisubharmonic.
There exists $\eps_0$ small enough so that for any $0 < \eps < \eps_0$ the intersection of the sphere of center $0$ and radius $\eps$ with $X$ is transversal, so that $\LC{\eps}(X,0) := X \cap \{ \rho = \eps\}$ is a real-analytic strongly pseudoconvex $3$-fold, and
bounds the Stein domain $X_\eps := X \cap \{\rho < \eps\}$.
Such $\eps_0$ may be taken maximal satisfying the above property, and any $\eps$ strictly smaller than the threshold $\eps_0$ above is said to be \emph{admissible}.

Observe that the diffeomorphism type of $\LC{\eps}(X,0)$ does not depend on the choice of an admissible $\eps$, but its embedding as a real-analytic $3$-fold in $X$ does.
In fact the diffeomorphism type of $\LC{\eps}(X,0)$ is also independent on the choice of an embedding in $\C^n$, see~\cite[Proposition~2.5]{MR3112993}. 


\subsection{Kato surfaces}\label{ssec:Katosurfaces}

A special case of real-analytic strongly pseudoconvex $3$-folds is given by \emph{spherical shells}, corresponding to the boundary of the unit ball $B$ in $\C^2$, i.e., to the link of a regular point $(X,0)$.
In \cite{kato:cptcplxmanifoldsGSS}, M. Kato considers the following construction to produce compact complex surfaces admitting a global spherical shell.
Let $Y$ be any connected open neighborhood of $\overline{B}$ in $\C^2$, 
and let $\pi\colon Y' \to Y$ be a proper bimeromorphic map which is an isomophism above $Y \setminus \{0\}$.
Let $y$ be a point in $\pi^{-1}(0)$, and pick a relatively compact neighborhood $U$ of $y$ in $\pi^{-1}(B)$ such that there exists a biholomorphism $\sigma\colon Y \to U$. 
We call the pair $(\pi,\sigma)$ a \emph{regular geometric Kato datum}, to distinguish this definition from the one given in \S\ref{sec:164}.
One can define a compact complex surface $S=S(\pi,\sigma)$, called \emph{Kato surface}, obtained from $Y' \setminus \sigma(\overline{B})$ by gluing $Y'\setminus \pi^{-1}(\overline{B})$ and $\sigma(Y\setminus \overline{B})$ using $\sigma \circ \pi$


Kato surfaces have been studied intensively in the literature, see for example the monograph of G. Dloussky~\cite{dloussky:phdthesis}. They are  compact complex surfaces with negative Kodaira dimension, $b_1 =1$, $b_2 >0$, and they admit a global spherical shell. 
 In the Kodaira classification of compact complex surfaces, they belong to the VII$_0$ class, and it is believed that they are the only examples of surfaces in this class having $b_2>0$, 
 see~\cite{MR2726099}.


\subsection{Proof of Theorem~\ref{thm3}}\label{ssec:proofthm3}
In this section we fix a sandwiched singularity $(X,x_0)$.  Our aim is to realize its archimedean link as a global real-analytic strongly pseudoconvex $3$-fold in a compact complex surface $S$ that has a global spherical shell.


Our first objective is to construct a complex analytic version of a Kato datum attached to  $(X,x_0)$.

Let $Y$ be a connected neighborhood of the closed unit ball in $\C^2$. 
By Artin's approximation theorem and \S \ref{ssec:sandwich}, we may find  a proper bimeromorphism $\pi \colon Y' \to Y$ that is an isomorphism over $Y\setminus \{0 \}$, and a non-empty connected divisor $D \subset \pi^{-1}(0)$ such that the normal complex analytic germ obtained by contracting $D$ to a point is isomorphic to $(X,x_0)$. 
We write $\mu \colon Y' \to X$ for the contraction morphism. 
It is a proper bimeromorphism that is a local isomorphism at any point of $D$ and contracts $D$ to the point $x_0$.

Observe that we may (and shall) assume that the support of $D$ 
is the union of all rational curves in  the exceptional divisor $\pi^{-1}(0)$ of self-intersection $\le -2$
by~\cite[Corollary 1.14]{spivakovsky:sandsingdesingsurfNashtransf}.
In that case  the map $\mu$ is the minimal good resolution of  $(X,x_0)$. Define the proper bimeromorphism $\eta \colon X \to Y$ by setting $\eta := \pi \circ \mu^{-1}$.

Consider any embedding of a local neighborhood of $x_0$ in $X$ into the unit ball in $\C^n$ that sends $x_0$ to the origin, so that we can talk about its complex link $\LC{\eps}(X,x_0)$ for any admissible $\eps \ll 1$. 
Recall the definition of the Stein domain $X_\eps$ from \S\ref{sec:pdcvx3folds}. 

Now pick any point $y \in \mu^{-1}(x_0)$, and choose a neighborhood $U$ of $y$ such that $\mu(U)$ is relatively compact in $X_\eps$ and there exists a biholomorphism $\sigma\colon Y \to U$ (we possibly have to shrink $Y$ in order to do that).
We construct a complex surface $X'$ by patching together $Y' \setminus \sigma(\overline{B})$ and $X$ 
using $\sigma \circ \eta \colon X \setminus \eta^{-1}(\overline{B}) \to \sigma(Y\setminus\overline{B})$.  


The identification of $X$ with its copy inside $X'$ induces a holomorphic map $\tilde{\sigma} \colon X \to X'$ that is a biholomorphism onto its image, and such that $\tilde{\sigma}(X)$ is relatively compact inside $X'$.
We also have a proper bimeromorphic map $\tilde{\pi} \colon X' \to X$ defined by $\tilde{\pi} = \mu$ on $Y' \setminus \sigma(\overline{B})$ and by $\tilde{\pi}= \mu \circ \sigma \circ \eta$ on $X$.

%
We proceed to construct a compact complex surface following Kato's construction as in the previous section. Pick three admissible positive real numbers $ \eps_- < \eps< \eps_+$, and set $X'_t = \tilde{\pi}^{-1}(X_t)$ for every $t$ in $\{\eps_-, \eps, \eps_+\}$. 
Define the surface $\tilde{S}$ considering $X'_{\eps_+} \setminus \tilde{\sigma}(\overline{X_{\eps_-}})$ and gluing together
$X'_{\eps_+} \setminus \overline{X'_{\eps_-}}$ and $\tilde{\sigma}(X_{\eps_+} \setminus \overline{X_{\eps_-}})$ via the map $\tilde{\sigma} \circ \tilde{\pi}$. 
Observe that the canonical map  $\overline{X'_\eps} \setminus \tilde{\sigma}(X_\eps) \to \tilde{S}$ is surjective
hence  $\tilde{S}$ is compact. 
It is a smooth surface, since the only singularity of $X'_{\eps_+}$ lies inside $\tilde{\sigma}(\overline{X_{\eps_-}})$.

%

By construction, $\tilde{S}$ contains a neighborhood of the archimedean link $\LC{\eps}(X,x_0)$, that is hence realized as a real-analytic strongly pseudo-convex $3$-fold bounding a Stein domain in $\tilde{S}$.

To see that $\tilde{S} \setminus\LC{\eps}(X,x_0)$ is connected, it is enough to see that $X'_{\eps} \setminus \overline{\tilde{\sigma}(X_{\eps})}$ is connected.
But $\tilde{\sigma}(\LC{\eps}(X,x_0))$ is a compact and connected real $3$-fold in $X'_{\eps}$, hence $X'_{\eps} \setminus \tilde{\sigma}(\LC{\eps})$ has at most two components, one of which is $\tilde{\sigma}(X_{\eps})$.
This implies that $X'_{\eps} \setminus \overline{\tilde{\sigma}(X_{\eps})}$ is connected.

It remains to see that $\tilde{S}$ contains a global spherical shell.
This follows from~\cite[Proposition~2]{kato:cpctcplxsurfwithGSPH}, whose proof is given on pp.~541--546. The proof should be read by replacing $Z(\delta)$ (resp. $A$, $0^*$, and $g$) by $X'_\eps$ (resp. by $\tilde{\pi}^{-1}(x_0)$, $\sigma(x_0)$ and $\tilde{\sigma} \circ \tilde{\pi}$).

%
%

One can also see it directly, by showing that $\tilde{S}$ is biholomorphic to the Kato surface $S=S(\pi,\sigma)$ associated with the regular Kato datum $(\pi,\sigma)$.
Define a map $\eta'\colon X'\to Y'$ by setting $\eta'=\id$ on $Y' \setminus \sigma (\overline{B})\subset X'$, and $\eta'=\sigma \circ \eta$ on $X\subset X'$.
Denote by $S^*$ the surface obtained from $X' \setminus \tilde{\sigma}\big(\eta^{-1} (\overline{B})\big)$ by gluing $X'\setminus \tilde{\pi}^{-1}\big(\eta^{-1} (\overline{B})\big)$
and $\tilde{\sigma}\big(Y\setminus \eta^{-1} (\overline{B})\big)$
using $\tilde{\sigma} \circ \tilde{\pi}$. 
The map $\eta'$ then induces a biholomorphism between $S^*$ and $S$.
We show that $S^*$ and $\tilde{S}$ are biholomorphic.

To that end, consider the surface $S'$ obtained from $X' \setminus \tilde{\sigma}(\overline{X_{\eps_-}})$ by gluing $X' \setminus \overline{X'_{\eps_-}}$ to $\tilde{\sigma}(X\setminus \overline{X_{\eps_-}})$ via the map $\tilde{\sigma} \circ \tilde{\pi}$.
The inclusions $\imath^*\colon X' \setminus \tilde{\sigma}(\eta^{-1} (\overline{B})) \hookrightarrow X' \setminus \tilde{\sigma}(\overline{X_{\eps_-}})$ and $\tilde{\imath}\colon X'_{\eps_+} \setminus \tilde{\sigma}(\overline{X_{\eps_-}}) \hookrightarrow X' \setminus \tilde{\sigma}(\overline{X_{\eps_-}})$ induce biholomorphisms $\Phi^*\colon S^* \to S'$ and $\tilde{\Phi}\colon\tilde{S} \to S'$.\hfill$\qed$

\subsection{Proof of Theorem~\ref{thm2}}

As above, we fix a closed embedding of $X$ in the open unit ball in $\C^n$, define $\rho (z) := \sum_{i=1}^n |z_i|^2$, and we write $\LC{\eps}(X,0):= X \cap \{\rho = \eps\}$ for any admissible $\eps>0$, $X_\eps:= X \cap \{ \rho < \eps\}$, and $A_{\eps_-,\eps_+} := X_{\eps_+} \setminus \overline{X_{\eps_-}}$.

We suppose that for some admissible $\eps$ the manifold $\LC{\eps}(X,0)$ can be realized as a global strongly pseudoconvex $3$-fold in a smooth compact complex surface $S$.
By assumption, there exist $\eps_- < \eps < \eps_+$ and a holomorphic embedding $\imath \colon A_{\eps_-,\eps_+} \to S$ such that $S\setminus \imath(A_{\eps_-,\eps_+})$ is connected.

Let $X'$ be the surface obtained by gluing $S\setminus \imath(\LC{\eps}(X,0))$ and $X_{\eps_+}$ together along the open subsets $\imath (A_{\eps,\eps_+})$ and $A_{\eps,\eps_+}$ via the map $\imath^{-1}$. 
This surface has a single singularity at the point $0'$ corresponding to the point $0$ of $X_{\eps_+}$, and $(X',0')$ is a normal singularity which is analytically isomorphic to $(X,0)$.
Denote by $\sigma$ the map induced by the identification of $X_{\eps_+}$ inside $X'$.
It is a holomorphic embedding mapping the singular point $0$ to $0'$.

Observe that there is a biholomorphic copy $A$ of $A_{\eps_-,\eps}$ that is included in $X'$ and whose complement in $X'$ is compact.
Since $X_{\eps}$ is Stein, the natural biholomorphism $A \to A_{\eps_-,\eps}$ extends to a proper bimeromorphic map $\pi \colon X' \to X_{\eps}$ by a classical theorem of Hartogs.
Note that $\pi$ could be a local isomorphism at $0$, and it does not necessarily send $0'$ to $0$.

\smallskip

If $\pi(0')\neq 0$, then $\pi$ is a proper bimeromorphic map from $(X',0')$ onto a smooth point in $X_{\eps}$,
therefore $(X',0')$ (and hence $(X,0)$) is a sandwiched singularity as was to be shown. 
We may thus assume that $\pi(0') = 0$ so that $0'$ is fixed by the map $f:=\sigma\circ \pi \colon X' \to X'$.

\smallskip

First suppose that $\pi$ does not induce a local biholomorphism from $(X',0')$ to $(X,0)$.
Then the maps $\pi \colon X' \to X_\eps$ and $\sigma\colon X_\eps \to X'$ define a Kato datum in the sense of page~\pageref{def:kato datum}.
We deduce that $(X,0)$ is sandwiched by the implication \ref{condition_kato}$\Rightarrow$\ref{condition_sandwiched} of Theorem~\ref{mainthm}.
The fact that $\tilde{S}$ contains a global spherical shell follows similarly as before from~\cite[Proposition~2]{kato:cpctcplxsurfwithGSPH}.

%

\smallskip

Now suppose that $\pi$ induces a local biholomorphism from $(X',0')$ to $(X,0)$.
Then the map $f$ induces a local biholomorphism from $(X',0')$ to itself and we have  
$K:=\bigcap_{n\in\N} f^n(\sigma(X_\eps)) = \{ 0'\}$. 
Indeed, $K$ is a compact subset of $\sigma(X_\eps)$, and, since the latter can be realized as a bounded set in $\C^n$, Montel's theorem applies and all limits of the sequence of iterates $\{f^m\}_{m\ge0}$ should be constant (this argument is due to M.~Kato, see the proof of \cite[Lemma 2]{kato:cptcplxmanifoldsGSS}). But $f$ is fixing $0'$, hence $K = \{ 0'\}$. 
In  other words, $f \colon (X',0')\to (X',0')$ is a contracting automorphism in the terminology of~\cite{MR3270424}.

Let $S(f)$ be the space of orbits of $f$, that is the quotient of $X'\setminus\{0'\}$ by the equivalence relation defined by $x \simeq x'$ if and only if $f^n(x) = f^m(x')$ for some positive $n,m\in\N$. This space is naturally a compact complex surface. Observe that one has a natural holomorphic map $S\setminus \imath\big(\LC{\eps}(X,0)\big) \subset X'$
to $S(f)$, and that this map actually descends to a map $S \to S(f)$.
It follows that $S$ is a modification of $S(f)$.  

The singularity $(X,0)$ and the geometry of $S(f)$ are completely described in~\cite[Theorem 7.5]{MR3270424}. 
We are thus in one of the following situations.

\begin{enumerate}[label=Case \arabic{enumi}.,leftmargin=0pt, itemindent=40pt]\setlength\itemsep{1mm}
\item The singularity $(X,0)$ is a cyclic quotient singularity. In this case \cite[Corollary~B]{MR3270424}
shows that $S(f)$ is a Hopf surface. More precisely the universal cover of $S(f)$ is isomorphic to $\C^2\setminus\{0\}$ and its fundamental group is the subgroup of polynomial automorphisms generated by $\gamma$ and $\tilde{f}$ in the notations of \cite[Example~7.1]{MR3270424}. 
In particular, this fundamental group is not cyclic, so that $S(f)$ is a secondary Hopf surface. We are in case \ref{item:thm2b} of the theorem.

\item \label{item:2} There exist a Riemann surface $C$, a line bundle $L\to C$ of negative degree, and a finite group $G$ of automorphisms of $C$ that acts linearly on $L$. Let $X'$ be the surface obtained by contracting the zero section in the total space of $L$, and let $0'$ be the image of the zero section in $X'$.
Then the germ $(X,0)$ is isomorphic to a neighborhood of the image of $0'$ in the quotient space $X'/G$. 
Moreover, one can find a positive integer $N\ge 1$ and a complex number $\alpha$ of norm smaller than $1$ such that $f^N$ lifts to a linear map acting by multiplication by $\alpha$ on the fibers of $L$.
By~\cite[Lemma 8.1]{MR3270424}, the natural map $S(f^N)\to S(f)$ is a Galois cyclic (unramified) holomorphic cover.

\smallskip

\begin{enumerate}[label=Case 2\alph{enumii}.,leftmargin=0pt, itemindent=45pt]\setlength\itemsep{1mm}
\item Suppose first that the genus of $C$ is positive. 
Then $(X',0')$ is not rational, hence $(X,0)$ is not rational either by~\cite[Claim 6.11]{MR1368632}. 
In particular $(X,0)$ is not a quotient singularity.  By~\cite[Theorem~A]{MR3270424}, it follows that  $(X,0)$ is weighted homogeneous.
Finally the proof of \cite[Corollary~B]{MR3270424} shows that  $S(f^N)$ is a principal elliptic fiber bundle of Kodaira dimension $0$ or $1$ so that we are in case \ref{item:thm2a} of our theorem.
\item Suppose that $C=\P^1$ is the Riemann sphere. Then $(X',0')$ is a cyclic singularity and 
 $(X,0)$ is a quotient singularity. Again the proof of \cite[Corollary~B]{MR3270424} shows that
$S(f)$ is a Hopf surface.
If the group $G$ acting on $(X',0')$ is trivial, then $(X,0)$ is a cyclic quotient singularity, and the surface is either a secondary Hopf or a primary Hopf.
In the latter case it contains a global spherical shell and we are in case \ref{item:thm2c}.
In the former case we are in case \ref{item:thm2b} of our theorem. 
When $G$ is non-trivial, then the fundamental group of $S(f)$ is not cyclic hence it is a secondary Hopf surface; and we are in case \ref{item:thm2b}. 
\end{enumerate}
\end{enumerate}
This completes the proof of Theorem~\ref{thm2}. \hfill$\qed$

\bibliographystyle{alpha}
\bibliography{biblio}

\vfill

\end{document}